\documentclass[11pt,a4paper]{amsart}

\usepackage{lib}
\usepackage{amssymb}
\usepackage{amsfonts}

\usepackage[colorinlistoftodos, textwidth=3cm, textsize=small]{todonotes}
\usepackage{booktabs}

\newcommand{\NN}{\mathbb{N}}
\newcommand{\ZZ}{\mathbb{Z}}
\newcommand{\QQ}{\mathbb{Q}}
\newcommand{\RR}{\mathbb{R}}

\DeclareMathOperator{\dist}{dist}

\newcommand{\floor}[1]{\left\lfloor #1 \right\rfloor}

\newcommand{\cc}{\mathbf{c}}
\newcommand{\ee}{\mathbf{e}}

\newcommand{\oo}{\mathbf{0}}
\newcommand{\pp}{\mathbf{p}}
\newcommand{\qq}{\mathbf{q}}

\newcommand{\uu}{\mathbf{u}}
\newcommand{\vv}{\mathbf{v}}
\newcommand{\xx}{\mathbf{x}}
\newcommand{\yy}{\mathbf{y}}

\newcommand{\vol}{\text{vol}}

\newcommand{\abs}[1]{{\left\lvert{#1}\right\rvert}}

\begin{document}

\title[Covering radii and the shifted Lonely Runner Conjecture]{Covering radii of $3$-zonotopes and the shifted Lonely Runner Conjecture
}
    \thanks{%
        Supported by grants PID2022-137283NB-C21 and PRE2022-000020, %
        funded by MCIN/AEI/10.13039/501100011033.%
    }%

\author[D.~Alc\'antara]{David Alc\'antara}
\author[F.~Criado]{Francisco Criado}
\author[F.~Santos]{Francisco Santos}

\address[D.~Alc\'antara, F.~Santos]{
    Departmento de Matem\'aticas, Estad\'istica y Computaci\'on\\
    Universidad de Cantabria\\
         Santander\\
         Spain}
\email{david.alcantara@unican.es,
francisco.santos@unican.es}

\address[F.~Criado]{Departamento de Matem\'aticas \\
    CUNEF Universidad\\
        Madrid \\
        Spain}
\email{francisco.criado@cunef.edu}

\maketitle

\begin{abstract}
    We show that the shifted Lonely Runner Conjecture (sLRC) holds for 5 runners.
    We also determine that there are exactly 3 primitive tight instances of the conjecture,
    only two of which are tight for the non-shifted conjecture (LRC).
    Our proof is computational, relying on a rephrasing of the sLRC in terms of covering radii of certain zonotopes (Henze and Malikiosis, 2017),
    and on an upper bound for the (integer) velocities to be checked (Malikiosis, Santos and Schymura, 2024+).

    As a tool for the proof, we devise an algorithm for bounding the covering radius of rational lattice polytopes, based on constructing dyadic fundamental domains.
\end{abstract}

\setcounter{tocdepth}{2}
\tableofcontents

\section{Introduction}\label{sec:introduction}
The lonely runner (LR) conjecture states that if $n+1$ runners run along a circle of
length one with constant, distinct, velocities, all starting at the origin, then for
every runner there is a time at which all other runners are at distance at least
$1/(n+1)$ from it.
The conjecture was posed in 1968 by J.~Wills~\cite{willslrc} in the language of
Diophantine approximation, and is currently proved up to $n=6$~\cite{7lrc}.
The conjecture has attracted some attention due to the simplicity of its statement and also because it admits several interpretations, from its original Diophantine approximation statement, to visibility obstruction, billiard trajectories or nowhere zero flows in graphs, among others.
See~\cite{perarnauserra2024thelonely} for a very recent survey.

We are interested in the so-called \emph{shifted} version of the conjecture,
a generalization in which each runner is allowed to have a different starting point. 
This version appeared in print for the first time in 2019~\cite{sLRC}, 
although the authors attribute it to Wills as well.
Both in the original version and in the shifted one, the runner we are looking at can be fixed at the origin, since only relative velocities are important.
Doing this, the shifted conjecture becomes the following:

\begin{conjecture}[Shifted Lonely Runner Conjecture (sLRC)]
\label{conj:slrc}
    Let $v_1, \dots, v_n$, $s_1, \dots, s_n \in \RR$ be real numbers, with the $v_i$ distinct and non-zero.
    % \footnote{In the original lonely runner conjecture dropping the condition that the $v_i$ be distinct is no loss of generality once we have fixed a runner with no speed, but the shifted version is false without it, even after fixing a runner.}
    Then, there is a $t \in \RR$ such that for every $i \in [n]$ one has $\dist(v_i t + s_i, \ZZ) \ge \frac{1}{n+1}$.
\end{conjecture}

Since the order of the runners is not relevant, we assume without loss of generality that
$v_1 < \ldots < v_n$.\label{above}

This shifted version of the Lonely Runner Conjecture is only currently proved up to
four runners~\cite{Cslovjecsek2022CovRadiusPolytope, Rifford2021sLRCFor3And4Runners}.
In this paper, we prove it for five runners.
Note that we consider this to be the case $n = 4$, while Rifford~\cite{Rifford2021sLRCFor3And4Runners}
refers to the case of four runners as $n = 4$, so what he proves is in fact our $n=3$.

For our proof, we use that in Conjecture~\ref{conj:slrc} (and in the original LR conjecture) there is no loss of
generality in assuming all velocities $v_i$ to be positive integers~\cite[Section~4.1]{Cslovjecsek2022CovRadiusPolytope}.
We then rely on the following result of Malikiosis, Santos and Schymura:~\begin{theorem}[\protect{\cite[Corollary 1.15]{Malikiosis2024LinExpCheckingLRC}}]
    \label{thm:thm_e}
Conjecture~\ref{conj:slrc} holds for $n = 4$ for all $v_1,v_2,v_3, v_4\in \ZZ_{>}0$
and $s_1,s_2,s_3,s_4 \in \RR$ satisfying with $\sum v_i \ge 196$.
\end{theorem}

That is, only the velocity vectors $(v_1, v_2, v_3, v_4) \in \ZZ^4$ with
$1 \le v_1 < v_2 < v_3 < v_4$ and $v_1 + v_2 + v_3 + v_4 \le 195$ need to be checked.
It is also obvious that dividing all velocities by a common factor $c$ does not change the problem, since the positions at time $t$ of the original problem coincide with the positions at time $ct$ of the new one.
With these considerations, our main result is:

\begin{theorem}
    \label{thm:main}
    There are 2\,133\,561 integer velocity vectors $(v_1, v_2, v_3, v_4) \in \ZZ^4$ with $1 \le v_1 < v_2 < v_3 < v_4$, $v_1 + v_2 + v_3 + v_4 \le 195$ and $\gcd(v_1, v_2, v_3, v_4) = 1$.
    Conjecture~\ref{conj:slrc} holds (for every choice of the $s_i$) for all of them.
\end{theorem}

\begin{corollary}
    \label{coro:main}
    The shifted Lonely Runner Conjecture (Conjecture~\ref{conj:slrc}) holds for five runners ($n=4$).
\end{corollary}

We also determine that there are exactly 3 primitive integer velocity vectors $(v_1, v_2, v_3, v_4)$
for which the bound $\tfrac{1}{5}$ in the conjecture is the best possible.
We call such vectors tight.

\begin{definition}
\label{def:tight}
A velocity vector $\vv = (v_1, \ldots, v_n) \in \ZZ^n$ is \emph{tight for the sLRC}
if there are starting points $s_1, \ldots, s_n \in \RR$ such that for every time $t\in \RR$,
there is an index $i\in [n]$ such that $\dist(v_i t + s_i, \ZZ) \leq \frac{1}{n+1}$; equivalently, such that
\[
v_i t + s_i \in \left[-\tfrac{1}{n+1},\tfrac{1}{n+1}\right] +  \ZZ.
\]
If this happens with $s_i=0$ for all $i$, then $v$ is \emph{tight for the (non-shifted) LRC}.
\end{definition}

Clearly, being tight for the nonshifted version implies tight for the shifted one.
As an example, it is easy to show, as was already observed by Wills~\cite{willslrcII},
that the vector $(1, \dots, n)$ is tight for the nonshifted version, for all $n \in \NN$: suppose that at a certain time $t$ we have no $v_i t$ in $[-1/(n+1), 1/(n+1)]+ \ZZ$.
Then the pigeon-hole principle implies that at that time at least two runners, say $i$ and $i'$ must be in the same interval $[j/(t+1),(j+1)/(t+1)]+ \ZZ$ for some $j\in [n-1]$.
But then, assuming w.l.o.g.\ that $i'>i$, we have the contradiction that%
\footnote{This argument is an instance of the general proof of Dirichlet's approximation theorem.}
\[
v_i t, v_{i'}t \in \left[\tfrac{j}{n+1},\tfrac{j+1}{n+1}\right]+\ZZ \quad
\Rightarrow
\quad
v_{i'-i} t = v_{i'}t- v_i t \in \left[-\tfrac{1}{n+1},\tfrac{1}{n+1}\right]+\ZZ,
\]

\begin{theorem}\label{thm:tight}
    The only integer velocity vectors $(v_1, v_2, v_3, v_4) \in \ZZ^4$ with
    $1 \le v_1 < v_2 < v_3 < v_4$, and $\gcd(v_1, v_2, v_3, v_4) = 1$ that are tight (for the shifted version)
    are $(1, 2, 3, 4)$, $(1, 3, 4, 6)$, and $(1, 3, 4, 7)$.
\end{theorem}

The vectors $(1, 2, 3, 4)$ and $(1, 3, 4, 7)$ are known to be tight for the non-shifted LRC and, in fact, they are the only ones for $n=4$~\cite{cusickviewobIII}.
The vector $(1, 3, 4, 6)$ is new and shows that shifted tightness does not imply the unshifted one.
We include a direct proof of its tightness in Proposition~\ref{prop:tight_1_3_4_6}.

\bigskip

Our proofs of Theorems~\ref{thm:main} and~\ref{thm:tight} (as well as the proof of Theorem~\ref{thm:thm_e} in~\cite{Malikiosis2024LinExpCheckingLRC})
are based on a rephrasing of the Lonely Runner conjecture (both shifted and non-shifted) in terms of lattice
$(n-1)$-zonotopes with $n$ generators~\cite{zonorunners,sLRC}.
For the shifted case this rephrasing goes as follows:
Conjecture~\ref{conj:slrc} holds for a certain $n$ if and only if every ``strong lonely runner zonotope'' with $n$ generators has its \emph{covering radius} upper bounded by $\tfrac{n-1}{n+1}$.
Here, a \emph{strong lonely runner zonotope} or \emph{sLR zonotope} for short is an $(n-1)$-dimensional lattice zonotope with $n$ generators and whose \emph{volume vector} has no repeated entries.
Via this rephrasing and Theorem~\ref{thm:thm_e} (or its slightly stronger version~\ref{thm:no_big_tight}), Theorems~\ref{thm:main} and~\ref{thm:tight} follow respectively from Theorems~\ref{thm:sLRZ_195} and~\ref{thm:sLRZ_tight_3}, which we prove computationally.

In Section~\ref{sec:zonotopes} we state more precisely the zonotopal rephrasing of sLRC (Proposition~\ref{prop:zono_sLRZ}) and recall some basic facts about the covering radius of a convex body $C$ and its relation to the existence of fundamental domains contained in dilations of $C$.
We also prove that the denominator of the covering radius of a lattice polytope can be explicitly
bounded in terms of the inequalities defining the polytope (Corollary~\ref{cor:cov_radius_size_bound}). This is needed in our algorithms, since it gives us some slackness in the inequalities needed to bound covering radii (Proposition~\ref{prop:cov_radius_safe_interval}).

In Section~\ref{sec:algorithms} we describe the algorithms that we use for our proofs.
Algorithm~\ref{alg:slrz_from_vv} constructs an sLR zonotope with small coordinates associated with a given velocity vector.
Algorithm~\ref{alg:decide_covering_radius} is our main algorithm, 
deciding whether the covering radius of a given rational polytope $P$ is bounded by a given rational number $\rho$.
The algorithm either constructs a \emph{dyadic fundamental domain} contained in $\rho P$, or certifies that such a domain does not exist.
The correctness of the algorithm (Theorem~\ref{thm:alg_decide_cov_radius_sLRZ_correct}) is based on the slackness given by the bounds of Section~\ref{sec:zonotopes} (Theorem~\ref{thm:cov_radius_dyadic}).
In Remark~\ref{rem:exact_mu}, we discuss how to turn Algorithm~\ref{alg:decide_covering_radius} into an algorithm that outputs the exact covering radius rather than only verifying a given bound.

In Section~\ref{sec:computational_results} we present our computational results.
More specifically, Section~\ref{subsec:implementation_details} comments on some implementation details, Section~\ref{subsec:computational} shows statistics on performance and complexity of the fundamental domains constructed, including geometric descriptions of the three tight instances, and
Section~\ref{subsec:certificates} describes what someone wanting to reproduce and/or independently validate our results would need to do.
To this end, we have published a \href{https://github.com/endorh/slrc-zonotopes}{git repository}\footnote{\url{https://github.com/endorh/slrc-zonotopes}}with all our code.
Besides the algorithms used for obtaining our results and general utilities to work with Lonely Runner zonotopes, the repository includes a certificate for each of the (more than two million) sLR $3$-zonotopes of Theorem~\ref{thm:main}, and a script that checks the correctness of all certificates in little time (half an hour on a standard PC for the whole list).
More information is available in the repository's README file, and in Section~\ref{subsec:certificates}.

\medskip
We finish this introduction proving from scratch the tightness of the velocity vector $(1, 3, 4, 6)$ of Theorem~\ref{thm:tight}.

\begin{proposition}[Tightness of $(1, 3, 4, 6)$]
    \label{prop:tight_1_3_4_6}
    Consider four runners in a circular track of unit length, with velocities $1$, $3$, $4$, and $6$.
    If the runners start at positions $s_1=0$, $s_3 \in \left[-\tfrac{1}{20}, \tfrac{1}{20}\right]$, $s_4=0$ and $s_6=\tfrac{1}{2}$, 
    then, at every point in time, there is some runner in the interval $\left[-\tfrac{1}{5}, \tfrac{1}{5}\right]$.
\end{proposition}

\begin{proof}
    Since all velocities are integers, the configuration is periodic with period one.
    To simplify notation, we multiply the length of the track by a factor of $60$, keeping the velocities unchanged, thus multiplying the period by $60$ as well.
    Observe that this changes the initial positions to $s_3\in [-3,3]$ and $s_6=30$.
    Here and in what follows, we are labeling runners by their velocities.

    Hence, we need to show that at every point in time between $0$ and $60$ some runner is in the interval $I\coloneqq[-12, 12]$ (modulo $60$).
    By symmetry around $t = 0$, we only need to consider $t$ between $0$ and $30$.

    \begin{itemize}
        \item $t\in \left[0, 12\right]$:
            Runner $1$ is at position $0+t \in I$.
        \item $t\in \left[12, 18\right]$:
            At $t=12$, runner $1$ is leaving $I$ but runner $4$ is at position $0+4\cdot 12=48=60-12$,
            so it enters $I$ and remains in it for the next $\frac{2\cdot 12}{4}=6$ units of time.
        \item $t\in \left[18, 23\right]$:
            At $t=18$, runner $4$ is leaving the interval $I$.
            Fortunately, runner $3$ is somewhere between $-3+3\cdot 18=51$ and $3+3\cdot 18=57$, so it lies in $I$.
            At $t=23$, runner $3$ is still between $-3+3\cdot 23 = 66$ and $3+3\cdot 23 = 72$, at worst leaving $I$.
        \item $t\in \left[23, 27\right]$:
            At $t=23$, runner $6$ is at position $30+ 6\cdot 23=168= -12\pmod{60}$, entering the interval $I$.
            It stays in $I$ for $\frac{2\cdot 12}{6}=4$ units of time, until $t=27$.
        \item $t\in \left[27, 30\right]$:
            At $t=27$, runner $6$ is leaving the interval $I$ but runner $4$ is at position $4\cdot 27=108=2\cdot 60-12$.
            Runner $4$ stays in $I$ for $\frac{2\cdot 12}{4}=6$ units of time, until $t=33$.
            \qedhere
    \end{itemize}
\end{proof}

We illustrate Proposition~\ref{prop:tight_1_3_4_6} and its proof in Figure~\ref{fig:diagram_tight_1_3_4_6}.

\begin{figure}[htb]
    \begin{tikzpicture}[xscale=0.1,yscale=0.5]

        % Runner 1: non-periodic on [0,12] and [48,60] at height 4
            % first pulse in the only period:
            \draw[very thick] (0,4) -- (12,4);
            % second pulse in the only period:
            \draw[very thick] (48,4) -- (60,4);
        % label
        \node[left] at (0,4) {Runner 1};

        % Runner 3: period 20, on from 0 to 4 and from 16 to 20, at height 3
        \foreach \n in {0,1,2} {
            % first pulse in each period:
            % \draw[very thick] ({20*\n+0},3) -- ({20*\n+4},3);
            % second pulse in each period:
            % \draw[very thick] ({20*\n+16},3) -- ({20*\n+20},3);

            % right half
            \fill[color=black] ({20*\n+0},2.85) -- ({20*\n+0},3.15) -- ({20*\n+4+1},3.15) -- ({20*\n+4-1},2.85) -- cycle;
            % left half
            \fill[color=black] ({20*\n+16-1},2.85) -- ({20*\n+16+1},3.15) -- ({20*\n+20},3.15) -- ({20*\n+20},2.85) -- cycle;
        }
        % label
        \node[left] at (0,3) {Runner 3};

        % Runner 4: period 15, on from 0 to 3 and from 12 to 15, at height 2
        \foreach \n in {0,1,2,3} {
            % first pulse in each period:
            \draw[very thick] ({15*\n+0},2) -- ({15*\n+3},2);
            % second pulse in each period:
            \draw[very thick] ({15*\n+12},2) -- ({15*\n+15},2);
        }
        % label
        \node[left] at (0,2) {Runner 4};

        % Runner 6: period 10, off from 0 to 3, on from 3 to 7, off from 7 to 10, at height 1
        \foreach \n in {0,1,2,3,4,5} {
            % single pulse for each period:
            \draw[very thick] ({10*\n+3},1) -- ({10*\n+7},1);
        }
        % label
        \node[left] at (0,1) {Runner 6};

        % Draw the time axis (common for all waves)
        \draw[->] (0,0) -- (62,0) node[right] {\(t\)};
            % Add tick marks at multiples of 10
            \foreach \x in {0,10,20,30,40,50,60} {
                \draw (\x,0.2) -- (\x,-0.2) node[below] {\small \x};
            }

        % Add vertical dotted lines at x=12 and x=48
        \draw[dashed, thick, draw=red] (12,0) -- (12,4.5);
        \draw[dashed, thick, draw=red] (48,0) -- (48,4.5);
        \draw[dashed, thick, draw=red] (27,0) -- (27,4.5);
        \draw[dashed, thick, draw=red] (33,0) -- (33,4.5);
        \draw[dotted] (18,0) -- (18,4.5);
        \draw[dotted] (42,0) -- (42,4.5);
        \draw[dashed, draw=blue] (23,0) -- (23,4.5);
        \draw[dashed, draw=blue] (37,0) -- (37,4.5);

        % Left and right borders
        \draw[draw=black] ( 0,0.2) -- ( 0,4.5);
        \draw[draw=black] (60,0.2) -- (60,4.5);
    \end{tikzpicture}

    \caption{A graphical version of the proof of Proposition~\ref{prop:tight_1_3_4_6}. The positions of the runner with velocity $3$ are represented by rhomboids so that
    each horizontal section represents a choice $s_3\in [-3,3]$ for its starting position.}
    \label{fig:diagram_tight_1_3_4_6}
\end{figure}
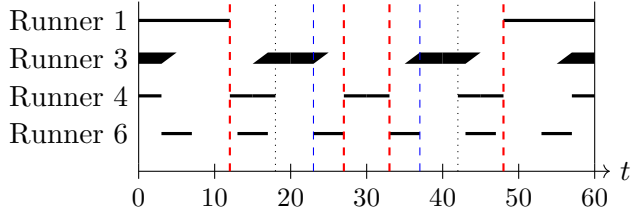

\begin{remark}
\label{rem:tight_1_3_4_6}
Although we do not prove it, the starting positions $(s_1,s_2,s_3,s_4)$ in the statement are \emph{exactly} the tight ones.
This can be derived from the description of the \emph{last covered points} of the corresponding sLR zonotope, which we include in Section~\ref{subsec:computational}.
Compare with Remark~\ref{rem:last_covered}.
\end{remark}

\section{Zonotopal statement of the Lonely Runner Conjecture}
\label{sec:zonotopes}

A \emph{zonotope} is any polytope obtained as the Minkowski sum of finitely many line segments.
As such, any zonotope $Z$ can be written as
\[
\cc + \sum_{i=1}^n [\oo, \uu_i] =\left\{\cc + \sum_{i=1}^n \lambda_i \uu_i : \lambda_i \in [0,1]\ \forall i\right\}
\]
for a certain finite set $\uu_1,\dots,\uu_n \in \RR^d$ of vectors, called the \emph{generators} of $Z$ and a certain point $\cc$.
The point $\cc$ is not important for us, since all we do is invariant under translation.
One natural choice is $\cc=\oo$ but often a more convenient choice is $\cc = -\tfrac12 \sum_{i=1}^n \uu_i$, since it
 makes the zonotope become $Z= \tfrac12\sum_{i=1}^n [-\uu_i, \uu_i]$ and be centrally symmetric around the origin.

\subsection{Lonely runner zonotopes and their volume vectors}
\label{subsec:lrz-volume-vector}

\begin{definition}[Lonely runner zonotopes, volume vector]
    \label{def:lrz_volume_vector}
    A \emph{Lonely Runner zonotope (LR zonotope)}
    is any zonotope $Z\subset \RR^{n-1}$ generated by a set of $n$ integer vectors $\mathbf U=\{\uu_i: 1 \le i \le n\}\subset \ZZ^{n-1}$ in linear general position;
    that is, such that every $n - 1$ of them are a linear basis of $\RR^{n-1}$.

    The \emph{volume vector} of $Z$ is the vector
    $\vv =(v_1, \dots, v_n)\in \ZZ_{>0}^n$ defined by
    \begin{equation}
         \label{eq:volume_vector_condition}
       v_i \coloneqq |\det(\mathbf U \setminus \{\uu_i\} )|.
    \end{equation} 
    When all entries of the volume vector are distinct, we say that $Z$ is a \emph{strong Lonely Runner zonotope (sLR zonotope)}.
\end{definition}

We call $\vv$ the {volume vector} of $Z$, because its entries are
the volumes of the $n$ parallelepipeds that make up $Z$.
In particular, we have that $\vol(Z) = \sum_{i=1}^n v_i$ (see details, e.g., in~\cite{Malikiosis2024LinExpCheckingLRC}).
The generators and the volume vector satisfy
\begin{equation} \label{eq:volume_vector}
v_1 \uu_1 \pm \cdots \pm v_n \uu_n =0
\end{equation}
for some choice of signs.
In fact, this equation (together with positivity) characterizes $\vv$ for given generators, up to a scalar factor.
Indeed, for each $j\in 1,\dots, n-1$, consider the determinant:
\[
    \begin{vmatrix}
      u_{1,j} & u_{2,j} & \cdots & u_{n,j}\\
      u_{1,1} & u_{2,1} & \cdots & u_{n,1} \\
      u_{1,2} & u_{2,2} & \cdots & u_{n,2} \\
      \vdots    & \vdots    & \ddots & \vdots    \\
      u_{1,n-1} & u_{2,n-1} & \cdots & u_{n,n-1} \
    \end{vmatrix},
\]
where $\uu_i=( u_{i,1},\dots,u_{i,n-1})$.
This determinant is zero because it has a repeated row, and it equals the $j$-th coordinate of~\eqref{eq:volume_vector} via Laplace's expansion of the determinant along its first row.
Since any $n-1$ of the vectors in $U$ are independent, Equation~\eqref{eq:volume_vector} has a $1$-dimensional space of solutions.
In this sense, Equation~\ref{eq:volume_vector} characterizes $\vv$ up to signs and scalar multiplication.

In the following result and the rest of the paper, a unimodular transformation is an
affine transformation with integer coefficients and determinant $\pm1$.
That is, an element of $\operatorname {AGL}(n,\ZZ) \coloneqq \ZZ^n \rtimes \operatorname {GL} (n,\ZZ)$.

\begin{proposition}[\cite{Malikiosis2024LinExpCheckingLRC}]
\label{prop:lSRZ_by_volume_vectors}
    For every integer vector $\vv=(v_1, \dots, v_n) \in \ZZ_{>0}^n$, there is some LR zonotope with
    integer generators and with volume vector $\vv$.
    If $\gcd(v_1, \dots, v_n) = 1$, then any two such zonotopes are equivalent by a unimodular
    transformation.
\end{proposition}

The following proof is paraphrased from~\cite{Malikiosis2024LinExpCheckingLRC}.
Algorithm~\ref{alg:slrz_from_vv} in Section~\ref{subsec:enumerating_slrzs} uses the ideas in the proof to
compute generators for a LR zonotope from its volume vector.

\begin{proof}
Start with the LR zonotope $Z'$ generated by $\uu'_i=-\ee_i$ for $i=1,\dots,n-1$ and $\uu'_n=\tfrac1{v_n}(v_1,\dots,v_{n-1})$.
It has volume vector $\tfrac1{v_n} \vv$.
Let $\Lambda'$ be the superlattice of $\ZZ^{n-1}$ generated by $\{\uu'_i: i \in [n]\}$ and let $T$ be a linear transformation bijecting $\Lambda'$ to $\ZZ^{n-1}$.
Then $T(Z')$ has volume vector $\tfrac1{m} \vv$, where $m=\gcd(v_1,\dots,v_n)$.
If $m\ne 1$, scale one coordinate up by a factor $m$ to get the correct volume vector.

To show uniqueness, let $Z$ and $Z'$ be two integer LR zonotopes both with volume vector $\vv$.
Let $\mathbf U=\{\uu_1,\dots,\uu_n\}$ and $\mathbf U'=\{\uu'_1,\dots,\uu'_n\}$ be their sets of generators.
Assuming $\gcd(v_1,\dots,v_n)=1$ implies that both $\mathbf U$ and $\mathbf U'$ span $\ZZ^n$ as a lattice.
Since changing the sign of a generator in a zonotope is a unimodular transformation (namely a translation), there is no loss of generality in assuming that
\[
    v_1 \uu_1+\cdots + v_n \uu_n = 0 = v_1 \uu'_1+\cdots + v_n \uu'_n.
\]
This
formula implies that the linear map sending $\uu_i$ to $\uu'_i$ for $i=1,\dots,n-1$ also sends $\uu_n$ to $ \uu'_n$, and the fact that $\mathbf U$ and $\mathbf U'$ span $\ZZ^n$ implies that this map is a unimodular transformation.
\end{proof}

\subsection{Covering radius and the sLRC}
\label{subsec:cov-radius-and-slrc}
A \emph{convex body} in $\RR^d$ is a convex compact subset.
We will always assume that our convex bodies are \emph{nondegenerate}, that is, that they have a non-empty interior.
This includes all bounded full-dimensional polytopes.

\begin{definition}[Covering radius]
    Let $C \subseteq \RR^d$ be a convex body.
    The \emph{covering radius of $C$}, denoted $\mu(C)$, is
    the smallest dilation factor $\rho > 0$ such that 
    \[
    \rho C + \ZZ^d = \RR^d.
    \]
\end{definition}

The covering radius is invariant under real translations and unimodular transformations of $C$, since they amount to translations and unimodular transformations of $\rho C + \ZZ^d$, respectively.

The following zonotopal restatement of the \emph{Shifted Lonely Runner Conjecture} is taken from~\cite{Malikiosis2024LinExpCheckingLRC} but the result is implicit in~\cite{zonorunners}, and also in~\cite{sLRC, Cslovjecsek2022CovRadiusPolytope}.

\begin{proposition}[\cite{zonorunners}, see also~\protect{\cite[Proposition 1.8]{Malikiosis2024LinExpCheckingLRC}}]
\label{prop:zono_sLRZ}
    Let $\vv = (v_1, \dots, v_n) \in \ZZ_{>0}^n$ with pairwise distinct entries
    and $\gcd(v_1, \dots, v_n) = 1$.
    Then, the following are equivalent:
    \begin{enumerate}
        \item The velocity vector $\vv$ satisfies the Shifted Lonely Runner Conjecture (Conjecture~\ref{conj:slrc}) for every choice of $(s_1,\dots,s_n)$.
        \item The sLR zonotope $Z$ with volume vector $\vv$ has $\mu(Z) \le \tfrac{n-1}{n+1}$.
    \end{enumerate}
%\pacoS{added the case of equality, although it is not explicitly stated anywhere (afaik)}
Moreover, the velocity vector $\vv$ is tight if and only if $\mu(Z) = \tfrac{n-1}{n+1}$.
\end{proposition}

The problem of exactly computing the covering radius of a rational polytope $P$ was first addressed by Kannan~\cite{Kannan1992LatTranslatesPolytope},
who proved it to be $NP$-hard and designed an algorithm that is doubly-exponential in the input size.
A better algorithm, singly exponential in the input size and polynomial (of degree $2d^3$) in fixed dimension $d$,
was designed in~\cite{Cslovjecsek2022CovRadiusPolytope}; we sketch another one in Remark~\ref{rem:exact_mu}.

But in order to use the zonotopal rephrasing to prove the sLRC one does not really need to
\emph{compute} the covering radius, only to check whether a certain number $\rho$ is a bound for it.
This checking is computationally much less expensive, and amounts to finding a fundamental domain inside a scaled copy $\rho P$ of $P$, or proving that none exists.

Recall that a \emph{fundamental domain} of $\RR^d$ (with respect to $\ZZ^d$) is a set containing exactly one representative of each coset $\pp + \ZZ^d$, $\pp \in \RR^d$.
The definition of covering radius easily translates to the following statements:

\begin{lemma}
    \label{lemma:cov_radius_ub_if_contains_fundamental_domain}
    \label{lemma:cov_radius_lb_iff_not_exists_open_representative_of_quotient}
    \label{lemma:cov_radius_lb_iff_not_exists_representative_of_quotient}
    Let $C$ be a convex body in $\RR^d$ and $\rho>0$.
    Then, the following are equivalent:
    \begin{enumerate}
        \item $\mu(C) \le \rho$.
        \item $\rho C$ contains a representative of each class in $\RR^d/\ZZ^d$.
        \item $\rho C$ contains a fundamental domain.
        \item There is no open set $W \in \RR^d$ with $\rho C \cap (W + \ZZ^d) = \emptyset$.
    \end{enumerate}
\end{lemma}

\begin{proof}
    The implications $(1)\Leftrightarrow(2)\Leftrightarrow(3)$ and $(2)\Rightarrow(4)$ are obvious.
    Let us prove $(4)\Rightarrow(1)$.
    $\mu(C) > \rho$ implies there is a $\pp \not\in \rho C + \ZZ^d$.
    Since $\rho C + \ZZ^d$ is closed, its complement $W \coloneqq \RR^d\setminus \rho C$ is open and
    \[
    \rho C \cap (W + \ZZ^d) \subset
    (\rho C \cap (\RR^d\setminus \rho C)) + \ZZ^d = \emptyset.
    \qedhere
    \]
\end{proof}

\medskip

Since tight instances of the sLR Conjeture correspond in Proposition~\ref{prop:zono_sLRZ} to sLR zonotopes $Z$ with $\mu(Z)$ exactly equal to $\tfrac{n-1}{n+1}$, the following result from~\cite{Malikiosis2024LinExpCheckingLRC} says that sLR $3$-zonotopes of volume at least 196 not only satisfy the sLR conjecture (Theorem~\ref{thm:thm_e}) but also that they are never tight:%
\footnote{This result is not in the currently public version of~\cite{Malikiosis2024LinExpCheckingLRC} (arXiv.v1) but it is to be included in a forthcoming version.}

\begin{theorem}[\cite{Malikiosis2024LinExpCheckingLRC}]
\label{thm:no_big_tight}
Every sLR $3$-zonotope $Z$ of volume at least $196$ has $\mu(Z) < \frac35$.
\end{theorem}

\subsection{The denominator of the covering radius}
\label{subsec:cov-radius-denominator}

It is well-known and easy to show that the covering radius of a rational polytope is rational (see, e.g.,~\cite[Proposition 5.1]{Kannan1992LatTranslatesPolytope}).
In this section we prove an explicit bound for its denominator in terms of the defining equations, where the denominator of a rational number $\rho$ is defined as the minimum positive integer $s$ such that $s\rho$ is an integer.
Having a bound for the denominator allows our algorithms to certify an exact upper bound for $\mu(P)$ from an approximate one, as follows.

\begin{proposition}
\label{prop:cov_radius_safe_interval}
    Let $P$ be a rational polytope and let $D\in \NN$ be an upper bound for the denominator of $\mu(P)$.
    (For example, but not necessarily, a bound obtained by Corollary~\ref{cor:cov_radius_size_bound} below).
    Let $\rho= r/s$ with $r,s\in \ZZ$ and $s>0$.
    Then, the following equivalences hold:
    \begin{enumerate}
        \item $\mu(P) \le \rho$ if and only if $\mu(P) < \rho + \frac{1}{sD}$.
        \item $\mu(P) \ge \rho$ if and only if $\mu(P) > \rho - \frac{1}{sD}$.
    \end{enumerate}
\end{proposition}

\begin{proof}
    One direction is obvious in both cases.
    For the other one, we know that $\mu(P) = \tfrac{r'}{s'}$ for integers $r',s'$ with $0< s'\le D$.
    Assuming $\tfrac{r'}{s'} \ne \tfrac{r}{s}$ we have that
    \[
        \abs{\mu(P) - \rho} = \abs{\tfrac{r'}{s'} - \tfrac{r}{s}} = \abs{\tfrac{r' s-rs'}{s's}} \ge \tfrac1{ss'} \ge \tfrac1{sD}.
    \]
    Hence, either $\mu(P) = \rho$, $\mu(P) \ge \rho + \frac1{sD}$ or $\mu(P) \le \rho - \frac1{sD}$.
\end{proof}

Our bound for the denominator uses the concept of last covered point.

\begin{definition}[Last covered point~\cite{CodenottiSantosSchymura, Cslovjecsek2022CovRadiusPolytope}]
    Let $C \subseteq \RR^d$ be a convex body.
    A \emph{last covered point} for $C$ is any $\pp \in \RR^d$ with
    $\pp \notin (\mu(C)\,C)^{\circ} + \ZZ^d$, where $C^\circ$ denotes the interior of $C$.
\end{definition}

Since $\mu(C)$ is invariant under translation,
we may assume $\oo \in C^\circ$ without loss of generality.
This simplifies some arguments because it implies that $\rho C +\qq$ is monotone increasing in $\rho$;
once a point is covered by some copy $\rho C + \qq$, it is also covered by $\rho' C + \qq$ for any $\rho' > \rho$.

The set of last covered points is always non-empty.
Indeed, assuming without loss of generality $\oo\in C^\circ$, for each point $\pp\in \RR^d$ let $\rho_{\pp} = \inf (\{\rho\in \RR_{\geq 0} : \pp \in (\rho C)^{\circ} + \ZZ^d\})$ be the \emph{covering time} of the point $\pp$.
The covering time is invariant by integer translations, and continuous.
Since $\RR^d / \ZZ^d$ is compact, there must be points $\pp$ with maximal covering time, i.e.\ $\rho_{\pp}=\mu(C)$.
These are precisely the last covered points.

In the rest of the section, $P\subseteq \RR^d$ is a polytope defined by the system of inequalities $Ax\le b$ for some matrix $A\in \RR^{m\times d}$ and vector $b \in \RR^m$.
For an element $i\in [m]$ or subset $I\subset [m]$, $A_i$, $b_i$, $A_I$, $b_I$, etc.
denote the restriction of a matrix or vector to the rows labeled by $i$ or $I$.

\begin{lemma}[\protect{\cite[Lemma 3.1]{Cslovjecsek2022CovRadiusPolytope}}]
\label{lemma:last-covered}
    Let $P=\{Ax\le b\}$ be a polytope with non-empty interior.
    Then, there is
    \begin{itemize}
        \item a subset $R\subset [m]$ of rows with $|R|=d+1$ and $\det(A_R|b_R) \ne 0$, and
        \item a lattice point $\qq_i\in \ZZ^d$ for each $i\in R$,
    \end{itemize}
    such that the system
    \begin{equation}
    \label{eq:last-covered}
      A_R (\xx-\qq_i) = t \,b_R
    \end{equation}
    has a unique solution in $\RR^{d+1}$ and this solution is of the form $(\xx,t) = (\pp,\mu(P))$, where $\pp$ is a last covered point.
\end{lemma}

\begin{proof}
  Let $\rho=\mu(P)$.
  As above, assume without loss of generality that $\oo\in P^\circ$, or
  equivalently, that $b_i > 0$ for all $i\in [m]$.

  Each facet of a translated polytope $\{\rho P+\qq: \qq \in \ZZ^d\}$ is labeled by a point $\qq\in \ZZ^d$ and an index $i\in [m]$.
  For each last covered point $\pp$, let $R_\pp$ be the set of indices $i$ such that $\pp$ lies in the $i$-th facet of $\rho P+\qq_i$ for some $\qq_i\in \ZZ^d$.
  Observe that $\pp$ lies in the affine subspace $L_\pp\coloneqq \cap_{i\in R_{\pp}}\{\xx\in \RR^d| A_i (\xx -\qq_i) = \rho b_i \}$, since
  $A_i (\xx -\qq_i) = \rho b_i$ is the facet equation for the $i$-th facet of $\rho P+\qq_i$.
  
  Choose a last covered point $\pp$ so that $R_\pp$ is maximal.
  Maximality implies that $L_\pp=\{\pp\}$, since otherwise, moving the point within $L_\pp$ until an extra facet of some $\{\rho P+\qq: \qq \in \ZZ^d\}$ is met (which happens at the latest when we
  are about to leave a certain $\rho P+\qq_i$ containing $\pp$), gives us a last covered point $\pp'$ with $R_{\pp'}$ strictly containing $R_\pp$.
  
  The fact that $L_\pp=\{\pp\}$ implies that the matrix $A_{R_\pp}$ consisting of rows used in $R_\pp$ has full rank, equal to $d$.

  Now, observe that the vectors $A_i$ for $i\in R_\pp$ must have a positive linear dependence.
  Otherwise, by (one of many versions of) Farkas' lemma, there is a vector $\vv\in \RR^d$ such that $\langle A_i, \vv\rangle >0$ for all $i\in R_\pp$.
  Then, $\pp$ would not be last covered, as $\pp+\varepsilon \vv$ would not be covered by any $P+\qq_i$.
  Locally, these are the only translated copies of $P$ that could potentially cover any point in a neighborhood of $\pp$ because $\pp$ cannot be in the interior of any translated copy of $P$.
  Therefore no translated copy of $P$ covers $\pp+\varepsilon \vv$, so $\pp+\varepsilon \vv$ would have larger covering time than $\pp$, contradicting the assumption of $\pp$ being last covered.

  The positive linear dependence among the vectors $\{A_i : i\in R_\pp\}$ implies that the system of equalities 
  \begin{align}
  \label{eq:system}
  A_i \,(\xx - \qq_i) = t b_i, \quad i \in R_\pp,
  \end{align}
  where $t$ is considered an extra variable, has no solution with $t\ne \rho$.
  Indeed, let $\lambda_i\in \RR_{\geq 0}$ for $i \in R_\pp$ be the coefficients of the linear dependence.
  Then,
  \begin{align}
    A_i\,(\xx-\qq_i) &= t b_i, & \forall i \in R_\pp,\nonumber \qquad \Rightarrow\\
    A_i\,(\xx-\pp) + A_i\,(\pp-\qq_i) &= \rho b_i + (t-\rho) b_i & \forall i \in R_\pp, \nonumber  \qquad \Rightarrow\\
    A_i\,(\xx-\pp) &=  (t-\rho) b_i & \forall i \in R_\pp, \nonumber  \qquad \Rightarrow
  \end{align}
  \begin{align}
   0 = \sum_{i\in R_\pp} \lambda_i A_i(\xx-\pp) & = \sum_{i\in R_\pp} \lambda_i (t-\rho) b_i % \nonumber  \\
    = (t-\rho) \left( \sum_{i\in R_\pp} \lambda_i b_i\right). \nonumber
  \end{align}

  But since the $\lambda_i$ are non-negative (and not all of them are zero) and the $b_i$ are positive, then it must be that $t=\rho$.

  Thus, the system~\eqref{eq:system} has only solutions of the form $(\xx, \rho)$.
  Since $A_{R_\pp}$ has rank $d$, it only has the solution $(\pp,\rho)$.
  This implies that the matrix $(A_{R_\pp} |$ $-b_{R_\pp})$ has rank $d+1$.
  Choose as $R$ any basis for the rows of $(A_{R_\pp} | b_{R_\pp})$.
\end{proof}

\begin{corollary}
    \label{cor:cov_radius_size_bound}
    Let $P$ be a rational polytope described by $A\,\xx\le b$ with $A\in \ZZ^{m\times d}$ and $b\in \ZZ^m$.
    Then $\mu(P)$ is a rational number and its denominator is bounded by
        \[\max_{R\in {[m]\choose d+1}} \abs{\det\left(A_R|b_R\right)}.\]
\end{corollary}

\begin{proof}
    Apply Cramer's rule to the variable $t$ in the system of Lemma~\ref{lemma:last-covered}.
\end{proof}

\begin{remark}
Cauchy-Binet implies that $\sqrt{\det\left((A|b)^T(A|b)\right)}$ is a more compact (but worse) bound.
Indeed:
    \[
        \max_{R\in {[m]\choose d+1}} \det(A_R|b_R)^2 \leq  \sum_{R\in {[m]\choose d+1}} \det(A_R|b_R)^2 =
        \det\left((A|b)^T(A|b)\right).
    \]
\end{remark}

\section{An algorithm to bound the covering radius of a rational polytope}
\label{sec:algorithms}

\subsection{Enumeration, construction, and preprocessing of sLR zonotopes}
\label{subsec:enumerating_slrzs}

According to Proposition~\ref{prop:zono_sLRZ}, to obtain Theorem~\ref{thm:main} and Corollary~\ref{coro:main}, we need to enumerate sLR zonotopes up to volume 195 and bound their covering radius by $3/5$.
We first construct the list of possible volume vectors, that is, the $4$-tuples $v=(v_1, \dots,v_4)\in \ZZ^4$ with $0<v_1<v_2<v_3<v_4$ and $\gcd(v_1,v_2,v_3,v_4)=1$.
Enumerating such $4$-tuples is algorithmically trivial and took less than a second in a standard PC:

\begin{proposition}
    \label{prop:volume_vector_count}
    There are exactly $2\,133\,561$ vectors $(v_1, v_2, v_3, v_4) \in \ZZ^4$ with $1 \le v_1 < v_2 < v_3 < v_4$, $\gcd(v_1, v_2, v_3, v_4) = 1$ and $\sum v_i \le 195$.
\end{proposition}

%\begin{algorithm}[H]
%    \caption{Enumerating strong Lonely Runner Zonotopes upto volume $V$}
%    \label{alg:enumerate_volume_vectors}
%
%    \DontPrintSemicolon
%    \SetKwInOut{Input}{Input}
%    \SetKwInOut{Output}{Output}
%
%    \Input{$V \in \ZZ_{>0}$}
%    \Output{$\textmono{volume\_vectors}$}
%
%    \SetKw{And}{and}
%
%    % Enumerate vectors
%    $\textmono{volume\_vectors} \gets \emptyset$
%    \For {$(v_1, v_2, v_3, v_4) \in \{1, \dots, V\}^4$} {
%        \uIf {$v_1 < v_2 < v_3 < v_4\ \And\ v_1 + v_2 + v_3 + v_4 \le V$} {
%            \uIf {$\gcd(v_1, v_2, v_3, v_4) = 1$} {
%                $\textmono{volume\_vectors} \gets \textmono{volume\_vectors} \cup \{(v_1, v_2, v_3, v_4)\}$\;
%            }
%        }
%    }
%    % We denote by [n] \coloneqq \{0, \dots, n-1\} \cap \ZZ
%\end{algorithm}

% With this algorithm we obtained:

To generate a representative zonotope from each volume vector $v$, we use Algorithm~\ref{alg:slrz_from_vv}, which follows the `existence' part of the proof of Proposition~\ref{prop:lSRZ_by_volume_vectors}.
Although we apply it to \emph{strong} lonely runner zonotopes, the algorithm works for arbitrary LR zonotopes, as long as the volume vector satisfies $\gcd(v_1,\dots,v_n)$=1.
That is, it needs the $v_i$ to be neither different, positive, nor ordered.

Step 1 in the algorithm creates an integer matrix $M'\in \ZZ^{(n-1)\times n}$ whose columns generate a LR zonotope with volume vector a scalar multiple of $(v_1,\dots,v_n)$.
Step 2 uses a column-wise Hermite normal form of $M'$ to construct a basis (the columns of the matrix $B$ in the algorithm) of the lattice $\Lambda$ generated by the columns of $M'$.
Observe that $\operatorname{rk}(M')=n-1$ implies that the last column of its Hermite normal form $H$ is zero, and $B$ is simply equal to $H$ without that column.

Now, $B^{-1}$ is the matrix of a linear isomorphism $\Lambda \xrightarrow{\cong} \ZZ^{n-1}$, so the columns of $B^{-1} M'$ would already be valid generators for a LR zonotope with volume vector $(v_1,\dots,v_n)$.
However, the generators obtained in this way typically have some large entries, resulting in `long and skinny' zonotopes that are poorly conditioned for computing their covering radii.
To overcome this, we preprocess the generators in step 3, by performing an LLL lattice basis reduction to the rows of $B^{-1} M'$.%
\footnote{We have implemented the LLL algorithm with $\delta=3/4$. Higher values of $\delta \in (0,1)$ would give better zonotopes, but would increase the running time.}
This produces a matrix $M$ whose columns are unimodularly equivalent to those of $B^{-1} M'$, but with smaller entries.

\begin{algorithm}
\SetAlgoHangIndent{0pt}
    \caption{Compute generators for a LR zonotope from its volume vector.}
    \label{alg:slrz_from_vv}

    \DontPrintSemicolon
    \SetKwInOut{Input}{Input}
    \SetKwInOut{Output}{Output}

     \Input{$v=(v_1, \dots, v_n)\in \ZZ_{>0}^n$, with $\gcd(v_1,\dots,v_n)=1$.}
    \Output{A matrix $M=(\uu_1 ,\dots,\uu_n)\in \ZZ^{(n-1)\times n}$ such that $\uu_1,\dots,\uu_n$ generate a LR zonotope with volume vector $v$.}

    \SetKw{And}{and}
    Let
   $
   M' \coloneqq
       \left(\begin{array}{c c c|c}
          -v_{n} &        &          & v_1    \\
                 & \ddots &          & \vdots \\
                 &        & -v_{n} & v_{n-1}
    \end{array}\right)
    \in \ZZ^{(n-1)\times n}.
    $\;
    Let $H\in \ZZ^{(n-1)\times n}$ be the column-wise Hermite normal form of $M$, and let $B \in \ZZ^{(n-1)\times (n-1)}$ consist of the first $n-1$ columns of $H$.

     Apply an LLL-reduction to the rows of $B^{-1} M'$ and let $M\in \ZZ^{(n-1)\times n}$ have as rows the resulting reduced vectors.\;
    \Return $M$.
\end{algorithm}

\medskip

For our covering radius algorithm, we need to convert the generators of the zonotope into an inequality description of it.
This, for an arbitrary zonotope $Z\subset \RR^d$ with generators $U=\{\uu_1,\dots,\uu_n\}$ is done as follows,
where we are identifying $\bigwedge^{d-1} \RR^d \cong (\RR^d)^*$ in the natural way.

\begin{proposition}
Let $Z= \tfrac12 \sum_{i=1}^n [-\uu_i,\uu_i]$ be the $\oo$-symmetric zonotope with generators $\uu_1,\dots,\uu_n$.
Then
\[
Z = \left \{\xx\in \RR^d : - {b_S} \le A_S \xx \le {b_S} : S \in \binom{[n]}{d-1} \right \},
\]
where 
\[
A_S \coloneqq \bigwedge_{i\in S} \uu_i \in (\RR^d)^*
\qquad\text{and}\qquad
b_S \coloneqq \frac{1}{2} \sum_{i=1}^n  \abs{A_S\, \uu_i}.
\]
\end{proposition}

\begin{proof}
Each facet of a zonotope is a zonotope itself, generated by the $\uu_i$ contained in, and spanning, a linear hyperplane.
Hence, every normal vector is indeed of the form $A_S $ for some $(d-1)$-subset $S$ of $U$.%
\footnote{%We are here identifying $\bigwedge^{d-1} \RR^d \cong (\RR^d)^*$ in the natural way. 
If the $\uu_i$ are not in general position, some of the $A_S$ may be zero, but that is not an issue; it just creates trivial inequalities.}

By central symmetry, there are two parallel facets with normal vectors $\pm A_S$.
The corresponding facet inequalities are $ - b_S \le A_S \xx \le b_S$, since $\pm b_S$ are the minimum and maximum values taken by $A_S$ in the set
\[
\left\{ \sum_{i=1}^n \pm \uu_i \right\},
\]
which contains all vertices of $Z$.
\end{proof}

\subsection{Certifying an upper bound for the covering radius}\label{subsec:certifying-cov-radius}
We here give an algorithm to decide whether a facet-defined polytope $P=\{\xx\in \RR^d : A\xx \le b\}$ contains a fundamental domain.
By Lemma~\ref{lemma:cov_radius_ub_if_contains_fundamental_domain}, this computational problem is equivalent to that of certifying a given upper bound $\rho$ for the covering radius of a polytope.

We consider a special family of fundamental domains of the integer lattice, given by unions of \emph{dyadic voxels}.

\begin{definition}
\label{def:dyadic}
A \emph{dyadic $d$-voxel} of \emph{level $\ell\in \ZZ_{\ge 0}$} is a half-open cube of the form
\[
\cc + \frac{1}{2^\ell} [0,1)^d,
\]
for some dyadic point $\cc\in \tfrac{1}{2^\ell} \ZZ^d$.
The integer point
$
\floor{\cc}\coloneqq(\floor{\cc_1},\dots, \floor{\cc_d}) \in \ZZ^d
$
 is the \emph{displacement} of the voxel, and the difference 
 \[
 2^\ell(\cc - \floor{\cc}) \in \{0,\dots, 2^\ell-1\}^d
 \]
  is the \emph{type} of the voxel.
 \end{definition}

All dyadic voxels of the same type are equivalent by integer translation, and the voxel types are naturally arranged as an infinite rooted $2^d$-ary tree with the voxels of level $\ell$ at depth $\ell$.
We call this the infinite dyadic tree.%
\footnote{One can represent each type of level $\ell$ as a vector $\tau=(\tau_1,\dots,\tau_d)$ where each $\tau_i$ is a binary string of length $\ell$. In this representation $\tau$ is an ancestor of $\tau'$ in the infinite dyadic tree if and only if each $\tau_i$ is an initial segment, or prefix, of the corresponding $\tau'_i$. Equivalently, if the voxel of type $\tau$ with zero displacement is contained in that of type $\tau'$.}

 A \emph{dyadic fundamental domain} is a fundamental domain obtained as a finite union of dyadic voxels. 
 Since two voxels of the same type are equivalent by integer translation, in order to decide whether a set of voxels defines a fundamental domain only their types matter, and no type can appear more than once in the set.
Moreover, we have the following equivalences:

\begin{proposition}
\label{prop:dyadic_fundamental_domains}
Let $S$ be a collection of dyadic $d$-voxels with no repeated types.
The following are equivalent:
\begin{enumerate}
    \item $S$ is a dyadic fundamental domain.
    \item The voxels of the types in $S$ with zero displacement cover $[0,1)^\ell$ without overlap.
          (We say that they form a dyadic subdivision of $[0,1)^\ell$).
    \item The nodes of the infinite dyadic tree corresponding to types in $S$ are the leaves of a full subtree containing the root.%
    \footnote{An $m$-ary rooted tree is called \emph{full} if every element has either $0$ or $m$ children.}
\end{enumerate}
\end{proposition}

\begin{proof}
The equivalence of (1) and (2) follows from the fact that voxels of the same type are equivalent by translation.
The equivalence of (2) and (3) is easy by induction on the number of internal nodes in the tree.
Induction starts with the tree with zero internal nodes, whose only leaf is the root, indeed a fundamental domain.
The inductive step comes from noting that substituting a leaf with its $2^d$ children amounts to subdividing the corresponding dyadic voxel.
\end{proof}

That is to say, every dyadic fundamental domain can be expressed as (the leaves of) a full-subtree of the infinite dyadic tree, with leaves labeled by their displacements.

This suggests a simple algorithm to construct a dyadic fundamental domain contained in a polytope $P$ via a search in the infinite dyadic tree, starting at the root:
If a node represents a type that fits (up to lattice translation) in $P$, add it to the fundamental domain you are constructing; if it doesn't, then add its $2^d$ children to the list of nodes to be processed.
This algorithm needs checking whether a given voxel admits a lattice translation contained in $P$, which amounts to the feasibility of an integer
linear program:
\begin{proposition}
    \label{prop:dyadic_voxel_in_polytope}
    Let $P=\{A\xx \le b\}\subset \RR^d$ be a polytope and let $V = \cc + [0, \epsilon)^d$ be a voxel.
    Then, $P$ contains an integer translation of $V$ if and only if the following Integer Linear Program is feasible:
    \begin{align}
    \label{eq:dyadic_voxel_in_polytope}
         \text{\rm find}       \;\;\xx \in \ZZ^d \qquad
         \text{\rm subject to} \;\;A \xx  \le b - A  \cc  - A_{\ge 0}\, {\boldsymbol \epsilon},
    \end{align}
    Here $\boldsymbol \epsilon \in \RR^d$ is the vector with all entries equal to $\epsilon$ and $A_{\ge 0}$ denotes the matrix with $(i,j)$-th entry equal to $\max\{0,A_{ij}\}$, for every $(i,j)$.
\end{proposition}

\begin{proof}
A half-open voxel is contained in $P$ if and only of the closed voxel is, that is, if and only of all the vertices of the translated closed voxel are in $P$.
Hence, for a given $\xx \in \ZZ^d$, the voxel $\xx + V =  \xx + \cc + [0, \epsilon)^d$ is contained in $P$ if and only if all the points $\yy\in \{0,\epsilon\}^d$ satisfy the inequalities $A(\yy+\cc +\xx) \le b$.
Now, for each row $A_i$ of $A$, the maximum value of the functional $A_i$ on the set $\{0,\epsilon\}^d$ is precisely $(A_i)_{\ge 0} \,{\boldsymbol \epsilon}$.
\end{proof}

Now, the search algorithm we have described has two issues:
on the one hand, if $P$ does not contain a dyadic fundamental domain then the search does not terminate;
on the other hand, this may happen even if $\mu(P) \le 1$, since $P$ may contain a fundamental domain but not a dyadic one.
For rational polytopes we can solve both issues resorting to Proposition~\ref{prop:cov_radius_safe_interval}.

\begin{theorem}
\label{thm:cov_radius_dyadic}
    Let $P$ be a rational polytope and let $D$ be an upper bound for the denominator of $\mu(P)$.
    Let $\rho=r/s$ with $r,s\in \ZZ$ and $s>0$.
    Let $P^+ = \left(\rho + \frac1{2sD}\right) P$.
    \begin{enumerate}
        \item If $\mu(P) \le \rho$, then $P^+$ contains a dyadic fundamental domain.
        \item If $\mu(P) > \rho$, then there is an $\ell \in \ZZ_{\ge 0}$ and a dyadic point $\cc\in \tfrac{1}{2^\ell}\{0,\dots, 2^\ell-1\}^d$ such that $P^+$ does not intersect $\cc +\ZZ^d$.
    \end{enumerate}
\end{theorem}

\begin{proof}
    For part (1) we use that
    \[
    \mu(P^+) = \frac{\mu(P)}{\rho + \frac1{2sD}} \le \frac{\rho}{\rho + \frac1{2sD}} <1.
    \]
    For each $\ell\in \NN$, let $D_\ell$ be the union of all the dyadic voxels of depth $\ell$ contained in $P^+$.
    Since $D_\ell$ converges (e.g., in the Hausdorff metric) to $P^+$ when $\ell$ goes to infinity, we have that $\mu(D_\ell)$ converges to $\mu(P^+)$.
    In particular, there is an $\ell$ such that $\mu(D_\ell) <1$.
    Hence, $D_\ell$ contains a fundamental domain, and this fundamental domain can be obtained taking one representative for each type of voxel in the union $D_\ell$.

    For part (2) we use that $\mu(P) > \rho$ implies (by Proposition~\ref{prop:cov_radius_safe_interval}) that $\mu(P) \ge \rho + \frac1{sD}$.
    Hence
    \[
    \mu(P^+) = \frac{\mu(P)}{\rho + \frac1{2sD}} \ge 
    \frac{\rho + \frac1{sD}}{\rho + \frac1{2sD}} >1.
    \]
    The statement then follows from the density of the dyadic
    points $\ZZ[\tfrac{1}{2}]^d$ in $\RR^d$ and Lemma~\ref{lemma:cov_radius_lb_iff_not_exists_open_representative_of_quotient},
    which asserts the existence of an open set $W \subset \RR^d \setminus (P^+ + \ZZ^d)$.
\end{proof}

The condition in part (2) can be tested for each $\cc$ with a variation of Eq.~\eqref{eq:dyadic_voxel_in_polytope} of Proposition~\ref{prop:dyadic_voxel_in_polytope}, simply taking $\epsilon=0$.
Since there are finitely many possible $\cc$ up to each $\ell$ and a dyadic fundamental domain has a bounded depth, Algorithm~\ref{alg:decide_covering_radius} successfully decides whether $\mu(P)\le \rho$, since it checks the cosets of all dyadic points for each $\ell$ in increasing order.

\begin{algorithm}[h]
    \caption{Decide whether $\mu(P)$ is bounded by $\rho$.}
    \label{alg:decide_covering_radius}

    \DontPrintSemicolon
    \SetKwInOut{Input}{Input}
    \SetKwInOut{Output}{Output}

    \Input{A rational polytope $P=\{A\xx\le b\}$ (with $A$ and $b$ integer) and a rational number $\rho=r/s$, with $r,s\in \ZZ$ and $s>0$.}
    \Output{A dyadic fundamental domain $S$ or a dyadic point $\cc$ certifying whether $\mu(P) \le \rho$ or not, as in Theorem~\ref{thm:cov_radius_dyadic}.}

    \SetKw{And}{and}
    \SetKw{Is}{is}
    Let $D$ be a bound on the denominator of $\mu(P)$, e.g., the one in Corollary~\ref{cor:cov_radius_size_bound}.\;
    Let 
    \[
        P^+=\left(\rho + \frac1{2sD}\right) P =\left\{A \xx \le \left(\rho + \frac1{2sD}\right) b\right\}
    \]\;
    Initialise a queue $N$ of `nodes to be processed' containing the root\;
    Initialise an empty list $S$ of `voxels in the fundamental domain'\; 
    \While {there are nodes in $N$} {
        Let $V = \cc +[0,\tfrac1{\ell^d})^d$ be one such node of minimum depth.\;
        Delete $V$ from $N$ and\;
        \eIf {$P^+$ does not intersect $\cc +\ZZ^d$\label{line:lb_test}} {
            \Return $\cc$
        } {
            \eIf { $\exists\; \pp\in \ZZ^d$ with $\pp +V \subset P^+$}
                {add the voxel $\pp + V$ to $S$}
                {add the $2^d$ children of $V$ to $N$\label{line:else}}
        }
    }
    \Return{$S$}
\end{algorithm}

Algorithm~\ref{alg:decide_covering_radius} is illustrated in Figure~\ref{fig:quad_tree_steps}.
\begin{figure}[b]
    \begin{minipage}{.25\textwidth}
        \centering
        \includegraphics[width=.90\textwidth]{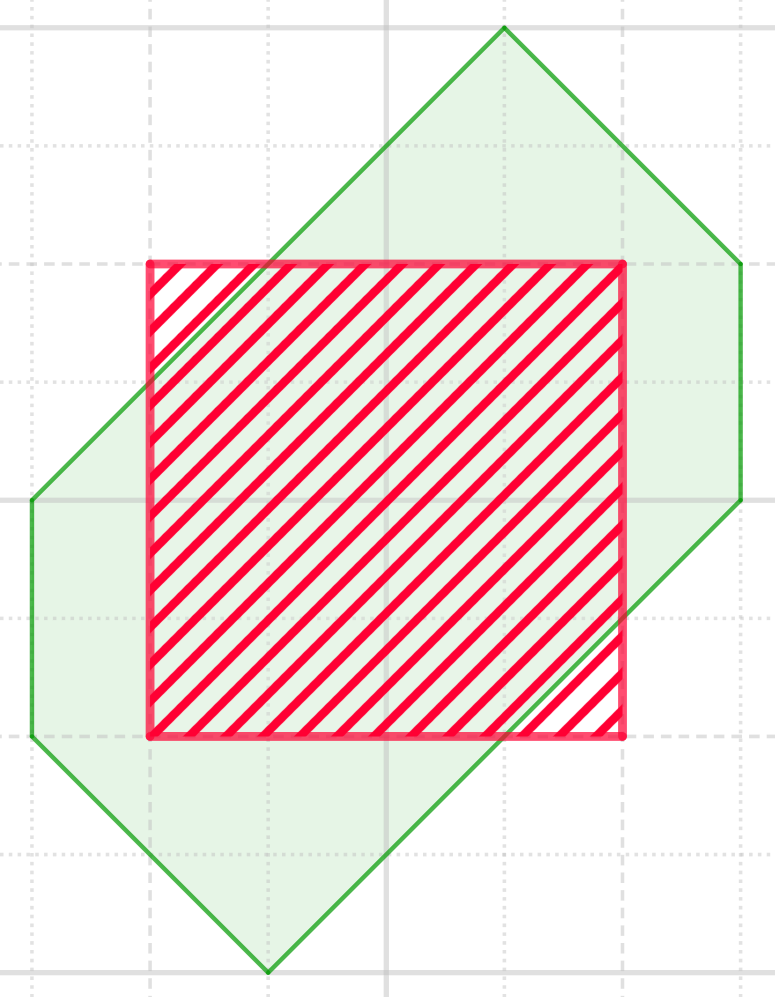}
    \end{minipage}\hfill%
    \begin{minipage}{.25\textwidth}
        \centering
        \includegraphics[width=.90\textwidth]{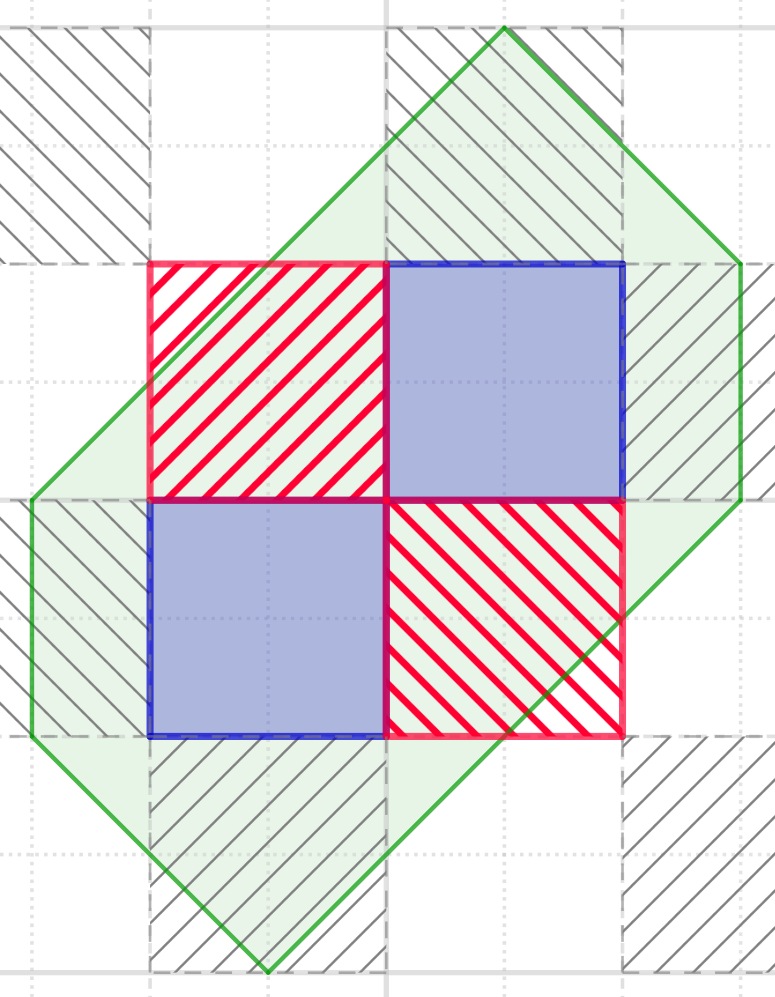}
    \end{minipage}\hfill%
    \begin{minipage}{.25\textwidth}
        \centering
        \includegraphics[width=.90\textwidth]{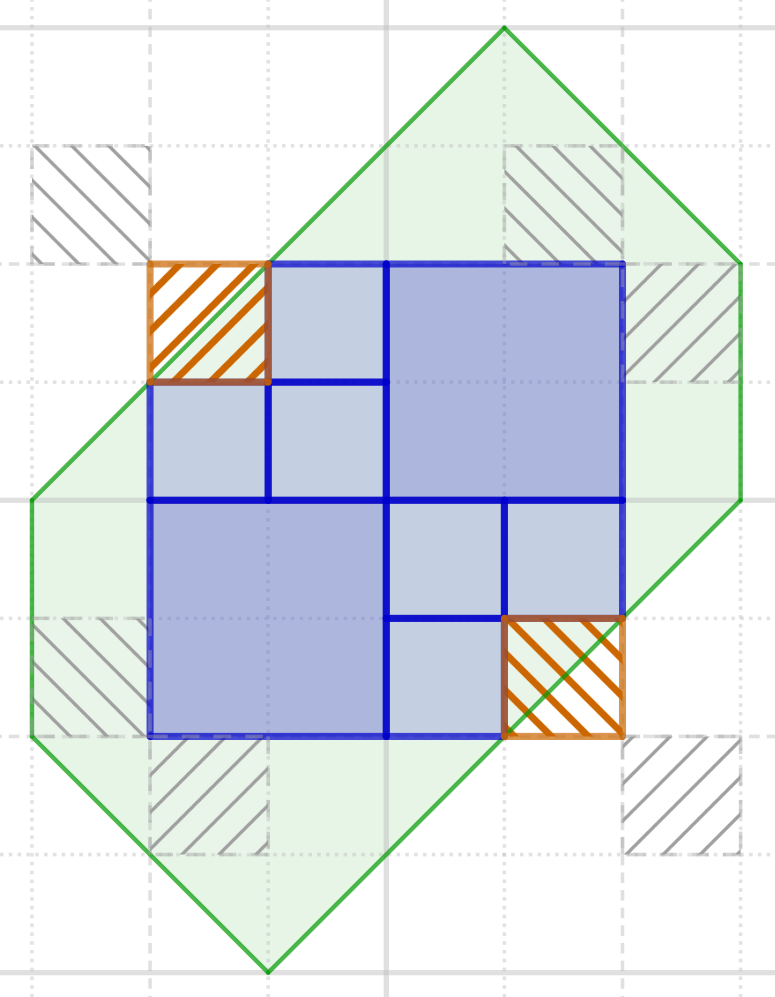}
    \end{minipage}\hfill%
    \begin{minipage}{.25\textwidth}
        \centering
        \includegraphics[width=.90\textwidth]{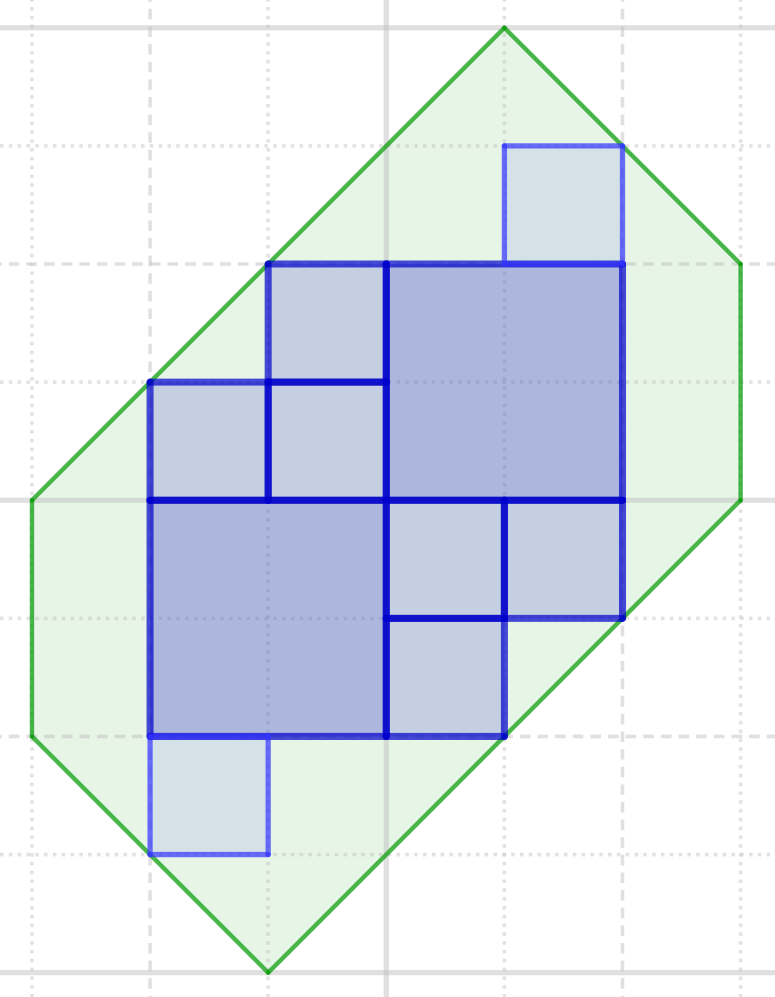}
    \end{minipage}\hfill
    \caption{States of Algorithm~\ref{alg:decide_covering_radius} at different dyadic depths, applied to $\frac12Z$ where $Z$ is the $2$-dimensional sLR zonotope with volume vector $(1,2,4)$. Notice that in the last step there are two choices for each voxel.
    (The figure takes $D=\infty$ for simplicity, although in general that could make the algorithm not terminate).}
    \label{fig:quad_tree_steps}
\end{figure}

\begin{theorem}
    \label{thm:alg_decide_cov_radius_sLRZ_correct}
    Algorithm~\ref{alg:decide_covering_radius} always terminates and it correctly decides whether
    $\mu(P) \le \rho$ for any lattice polytope $P$ and
    $\rho \in \QQ$.
\end{theorem}

\begin{proof}
    Observe that the algorithm returns a certificate in either case.
    Let us first show their correctness.

    If the algorithm finishes with a set of dyadic voxels, these voxels
    are a full subtree of the infinite dyadic tree by construction, and
    hence they form a dyadic fundamental domain.
    Furthermore, all of these voxels are contained in $P^+$, so
    $\mu(P) \le (\rho + \frac1{2sD})$ and
    Proposition~\ref{prop:cov_radius_safe_interval} implies $\mu(P) \le \rho$.

    On the other hand, if the algorithm finishes with a point
    $\cc$ such that $P^+$ does not intersect $\cc + \ZZ^d$,
    Lemma~\ref{lemma:cov_radius_lb_iff_not_exists_representative_of_quotient}
    implies $\mu(P) > (\rho + \frac1{2sD}) > \rho$.

    To prove that the algorithm terminates, we handle the two cases separately.

    If $\mu(P) \le \rho$, Theorem~\ref{thm:cov_radius_dyadic} guarantees the existence of a dyadic fundamental domain $D$ contained in $P^+$.
    Let $\ell$ be the maximum depth of the voxels in $D$.
    Then every voxel type of depth $\leq \ell$ has a representative contained in $P^+$, so the algorithm will never enter the ``else'' in line~\ref{line:else} with a voxel of depth $\leq \ell$.
    Hence, the algorithm can perform the while loop only finitely many times before $N$ becomes empty.

    If $\mu(P) > \rho$, Theorem~\ref{thm:cov_radius_dyadic} guarantees the existence of a dyadic point $\cc$ with $(\cc + \ZZ^d)\cap P^+=\emptyset$.
    Let $\ell$ be the minimal depth of such a point.
    Since the algorithm processes the infinite dyadic tree in a breadth-first search manner, it will check in a finite number of steps all the dyadic points of depth $\ell$ (either implicitly for those contained in voxels of depth $\le \ell$ and with $\pp + V \subset P^+$, or explicitly for those not contained in such voxels).
\end{proof}

% \begin{remark}
%     \label{rmk:dfs_termination}
%     Note how the proof of termination only requires that the algorithm explores
%     the dyadic tree in breadth-first order for the case where $\mu(P) > \rho$.
%     A depth-first search may only fail to terminate in such cases.
% \end{remark}

\begin{remark}
    \label{rmk:alg_decide_cov_radius_ge_sLRZ}
    Changing $\rho + \tfrac{1}{2dD}$ to $\rho - \tfrac{1}{2dD}$ in Algorithm~\ref{alg:decide_covering_radius} results in an
    algorithm that decides whether $\mu(P) \ge \rho$, with the same proof of correctness and termination:
    If the algorithm returns a dyadic fundamental domain then $\mu(P) < \rho$, otherwise
    Proposition~\ref{prop:cov_radius_safe_interval} implies $\mu(P) \ge \rho$.
\end{remark}

\begin{remark}
\label{rem:exact_mu}
    \label{rmk:alg_compute_cov_radius}
    Algorithm~\ref{alg:decide_covering_radius} can be adapted to 
    compute the exact covering
    radius of an arbitrary rational polytope $P$.
    We still perform a search in the dyadic tree so that at every step the leaves of the already explored tree form a
    dyadic fundamental domain, and then we iteratively subdivide leaves corresponding to voxel types that do not fit in the polytope, but we additionally do the following:
    \begin{itemize}
    \item For each voxel type $V$ (i.e., node in the dyadic tree) that we process we derive two bounds:
    A lower bound $\alpha_V$, given by the minimum $\rho$ such that the center of 
    some voxel of type $V$ is contained in $\rho P$, and an upper bound $\beta_V$
    given by the minimum $\rho$ such that some voxel of type $V$ is contained within $\rho P$.%
    \footnote{Both of these bounds are rational, but computing them precisely
    requires solving a Mixed Integer Linear Program with exact arithmetic related to the one of Proposition~\ref{prop:dyadic_voxel_in_polytope}}

    We let $\alpha =\max_V \alpha_V$ and $\beta = \max_V \beta_V$ where the $\max$ are
    taken over all the voxels in the current fundamental domain, so that $\mu(P)$ is guaranteed to lie in the interval $I=[\alpha, \beta]$.

    \item Instead of exploring the tree in breadth-first order, whenever we need to subdivide
    a voxel we do so with the one with the maximum $\beta_V$.\footnote{For efficiency, rather than looking for the maximum $\beta_V$ each time, one can maintain a priority queue of the voxel types.}

    \item At every step we let $\rho_I $ be the point of $I$ with the smallest denominator%
    \footnote{This point can be obtained as the highest node of the Stern-Brocot tree that is contained in the search
    interval $I$. For efficiency, rather than computing $\rho_I$ from scratch one can just descend in the Stern-Brocot tree from one $\rho_I$ to the next whenever $I$ changes.}.
    This $\rho_I$ is a candidate to be the exact covering radius, and it is declared to coincide with $\mu(P)$
    when the bounds from Proposition~\ref{prop:cov_radius_safe_interval}
    applied to $\rho_I$ discard all other elements of the interval $I$ from being the covering radius of $P$.
    \end{itemize}
        The algorithm returns the final $\rho_I=\mu(P)$ together with: the interval $I=[\alpha,\beta]$;
    the dyadic fundamental domain constructed, certifying that $\beta$ is an upper bound for $\mu(P)$;
    and the center of the voxel type with $\alpha_V = \alpha$, certifying that $\beta$ is a lower bound.
\end{remark}

\section{Computational Results}
\label{sec:computational_results}
\subsection{Implementation Details}
\label{subsec:implementation_details}

We have implemented Algorithms~\ref{alg:slrz_from_vv} and~\ref{alg:decide_covering_radius} in Python and have run them over all the volume vectors of Proposition~\ref{prop:volume_vector_count}.
The implementation is available at the \href{https://github.com/endorh/slrc-zonotopes}{git repository}%
\footnote{\url{https://github.com/endorh/slrc-zonotopes}}.

Our implementation combines rational arithmetic with floating point arithmetic.
We allow ourselves to use floating point arithmetic (namely the HiGHS MIP solver~\cite{Huangfu2017}) to find feasible solutions of the problems in Proposition~\ref{prop:dyadic_voxel_in_polytope}.
This is subject to numerical errors,
but in order to guarantee our computations, we then round all the solutions proposed by the MIP solver and verify their feasibility in exact arithmetic.
In solving the ILPs needed to construct certificates for all sLR zonotopes with volume up to $195$
we encountered no instance of the solver providing an invalid numerical solution.%
\footnote{
    However, experimenting with zonotopes of larger volume (up to $1000$), we did find various
    instances where the solver proposed solutions which were infeasible after rounding.
In such cases, in the absence of an exact MIP solver, a brute force approach could be
used, checking all candidate translations within the bounding box of the zonotope.}

Since our zonotopes are centrally symmetric around the origin, a voxel type will lie in our fundamental domain if and only if the opposite type does.
Hence, we only need to check half of the voxels in the first subdivision of the
unit cube, which automatically gives that we check only half of each level.
This has the advantage of halving the execution time and producing centrally symmetric
certificates, which are both smaller and visually clearer.

Furthermore, as we are interested in characterizing which velocity vectors are tight
for the sLRC, that is, which sLR zonotopes have covering radius exactly $\tfrac{3}{5}$,
we have also implemented and run the algorithm from Remark~\ref{rmk:alg_decide_cov_radius_ge_sLRZ};
this amounts just to a sign change in Algorithm~\ref{alg:decide_covering_radius}.

\begin{remark}
\label{rem:bound_in_negative_case}
Since very few sLR zonotopes are tight, we run the modified algorithm of Remark~\ref{rmk:alg_decide_cov_radius_ge_sLRZ} first which, in the non-tight case, certifies $\mu(Z) < \tfrac{3}{5}$.
Only in the (three) instances where the modified algorithm says $\mu(Z) \ge \tfrac{3}{5}$
we run the original algorithm to verify that $\mu(Z) \le \tfrac{3}{5}$.
%This avoids checking most zonotopes twice.

To save additional running time, when running the modified algorithm,
we first attempt a choice of $D$ smaller than that of Corollary~\ref{cor:cov_radius_size_bound},%
\footnote{We use $\sqrt{\det{A^T A}}$.}
as a smaller bound yields easier ILPs and the implication $\mu(Z) \le \rho - \epsilon \Rightarrow \mu(Z) \le \rho$ does not need $\epsilon$ to be small.
Only when the algorithm with the smaller $D$ fails to prove $\mu(Z) \le \rho - \epsilon$, we rerun with the guaranteed $D$ of the corollary.
\end{remark}

\begin{remark}
Dyadic fundamental domains with small circumradius or small volume of convex hull can be obtained with the same algorithm, but turning
the feasibility problem~\eqref{eq:dyadic_voxel_in_polytope} from Proposition~\ref{prop:dyadic_voxel_in_polytope} into an optimization problem that minimizes some norm.
This modification does not affect the search strategy or the types of voxels obtained in the final fundamental domain; it just gives the ``best'' representative of each type.

For example, the optimization problem for the Minkowski norm of $P$ is particularly simple:
\begin{align*}
    \label{eq:dyadic_voxel_in_polytope_minkowski_norm}
    \text{\rm minimize}   \;\;&\rho\in \RR\\
    \text{\rm subject to} \;\;&A \xx - b\rho \le - A  \cc  - A_{\ge 0}\, {\boldsymbol \epsilon}\\
                              &\rho \ge 0\\
                              &\xx \in \ZZ^d.
\end{align*}
Optimizing with respect to this norm results in a dyadic fundamental domain fitting
in the smallest possible dilation of the zonotope, among those with the types given by the breadth-first search.
% \pacoS{changed ``contraction'' to ``dilation''. According to wikipedia, one can use dilation for an homothety even if the ratio is smaller than one. Contraction can be confused with the oriented matroid contraction, which projects a zonotope along a generator}

The convex hull of the domain can be minimized even more by further subdividing all voxel
types to reach a regular tree of any given depth, at the expense of solving many more ILPs.
\end{remark}

\subsection{Computational results}
\label{subsec:computational}
Our main computational result is the following, from which Theorem~\ref{thm:main}
follows.

\begin{theorem}
    \label{thm:sLRZ_195}
    All 2\,133\,561 sLR $3$-zonotopes with integer volume vectors satisfying $1 \le v_1 < v_2 < v_3 < v_4$ and
    $\gcd(v_1, v_2, v_3, v_4) = 1$ up to $\sum v_i \le 195$ have covering radius less than or equal to $\tfrac{3}{5}$.
\end{theorem}

Our computational proof is a list of the stated sLR zonotopes, each paired with a
bound $D$ for the denominator of its covering radius, and a dyadic fundamental domain
contained within the zonotope dilated by a factor of $\tfrac{3}{5} + \tfrac{1}{2\cdot 5\cdot D}$.
Details on these certificates and how they can be verified are discussed in
Section~\ref{subsec:certificates}.

\subsubsection*{Performance Analysis}
Running Algorithm~\ref{alg:slrz_from_vv} to
build zonotopes for the 2\,133\,561 volume vectors up to volume $195$ took about 10 minutes of
computing time in a standard PC, without any parallelization.
%We have not run it over the $1,578$ million up to volume $1000$.
Constructing dyadic fundamental domains for all of them 
with Algorithm~\ref{alg:decide_covering_radius} and/or its variant of Remark~\ref{rmk:alg_decide_cov_radius_ge_sLRZ} took around 15 minutes in the same setting.

Most of the time in Algorithm~\ref{alg:slrz_from_vv} was spent by the LLL step that we use to make our zonotopes ``rounder''.
The fact that globally this time is similar to that of Algorithm~\ref{alg:decide_covering_radius} indicates that 
the value of $\delta=3/4$ that we used for LLL is adequate: a higher value would increase the time spent in Algorithm~\ref{alg:slrz_from_vv} and a lower one will give worse-conditioned zonotopes, increasing the time for Algorithm~\ref{alg:decide_covering_radius}.

\begin{table}[htb]
    \centering%
    % CSV:
    % Dyadic depth, Frequency
    % 0,2118699
    % 1,8889
    % 2,3958
    % 3,1991
    % 4,18
    % 5,3
    % 10,1
    % 11,1
    % 12,1
    \begin{tabular}{rcccccccccc}
        \toprule
           \textbf{Dyadic depth} &        0 &     1 &     2 &     3 &   4 &  5 &&  10 &  11 &  12\\
        \midrule
           \textbf{Frequency}    &  2118699 &  8889 &  3958 &  1991 &  18 &  3 &&   1 &   1 &   1\\
        \bottomrule
    \end{tabular}
    \medskip
    \caption{Depth of the certificates. The three significantly deeper instances correspond to the tight sLR zonotopes}
    \label{tab:slrzs_data_195}
\end{table}
We have gathered the distribution of depths of our dyadic fundamental domain certificates in Table~\ref{tab:slrzs_data_195}.
Significantly, 99.30\% of the certificates have depth $0$.
In them, Algorithm~\ref{alg:decide_covering_radius} finishes right at the root of the dyadic tree simply asserting that the centered unit cube is contained in the computed scaled zonotope.
This stresses the importance of conditioning the zonotopes via the LLL algorithm before computing a fundamental domain, since a preliminary test revealed that without the LLL conditioning fewer than 30\% of the scaled zonotopes contained a unit cube.

% With our initial choice of $\delta = \tfrac{3}{4}$, the LLL computation
% takes a similar amount of time to the construction of dyadic
% fundamental domains, so neither process is a bottleneck.
%
% However, this observation tells little about the scalability of our method
% to higher dimensions for two reasons.
% Firstly, LLL has an asymptotic complexity with respect to the dimension
% of $O(d^6)$, while a non-rigorous argument suggests
% Algorithm~\ref{alg:decide_covering_radius} has worst case complexity
% $O(2^{d(d-1)})$, suggesting the later should become more problematic in higher dimensions.
% This follows from the observation that the bound from Corollary~\ref{cor:cov_radius_size_bound}
% grows as $O(B^d)$ where $B$ is the bound on the entries of the matrices $A$ and $b$,
% meaning the depth reached by the worst case grows as $O(d)$, while on every
% step of the algorithm, the asymptotic number of children that need to be
% subdivided again is limited to the intersection with a $d-1$-dimensional
% subspace (a facet).
% Lastly, the worst case complexity hardly tells the whole picture, as,
% even in our case, most of the time is spent in the enumeration and rapid
% discarding of the exponentially many trivial cases.
% It is hard to judge which of these factors will become the main bottleneck
% in higher dimensions.

% A compromise that would avoid having to answer this question is running LLL
% in parallel to the search of dyadic fundamental domains, restarting the later
% when a more suitable representative has been found.

\subsubsection*{The three tight sLR zonotopes for $n=4$}\label{subsubsec:tight_slr_3_zonotopes}
As already mentioned, our computations show that only three sLR $3$-zonotopes with volume at most
$195$ are tight.
More precisely:

\begin{theorem}
    \label{thm:sLRZ_tight_3}
    The only volume vectors $(v_1, v_2, v_3, v_4) \in \ZZ^4$ with $1 \le v_1 < v_2 < v_3 < v_4$ and
    $\gcd(v_1, v_2, v_3, v_4) = 1$ up to $\sum v_i \le 195$ whose
    associated sLR zonotopes have covering radius equal to
    $\tfrac{3}{5}$ are $(1, 2, 3, 4)$, $(1, 3, 4, 6)$ and $(1, 3, 4, 7)$.
    Small generators for these sLR zonotopes are:
    \newcommand{\php}{\hphantom{+}}
    \begin{enumerate}
        \item $v=(1, 2, 3, 4)$:\; $\{ (1\;   {-}1\; 1),\; (0\; \php 1\; \php 1),\; (\php 1\; \php 1\;   {-}1),\; (  {-}1\;   {-}1\; \php 0)\}$
        \item $v=(1, 3, 4, 6)$:\; $\{(1\; \php 1\; 2),\; (1\;   {-}1\; \php 0),\; (  {-}1\;   {-}1\; \php 1),\; (\php 0\; \php 1\;   {-}1)\}$
        \item $v=(1, 3, 4, 7)$:\; $\{(1\; \php 0\; 2),\; (1\;   {-}1\;   {-}2),\; (  {-}1\;   {-}1\; \php 1),\; (\php 0\; \php 1\; \php 0)\}$
    \end{enumerate}
\end{theorem}

Together with Theorem~\ref{thm:no_big_tight} and the equality case of Proposition~\ref{prop:zono_sLRZ} this proves Theorem~\ref{thm:tight}.
Cusick and Pomerance~\cite{cusickviewobIII} have shown that $(1, 2, 3, 4)$ and $(1, 3, 4, 7)$ are
the only tight (primitive) instances of the non-shifted LRC for $n = 4$, and
tightness of a vector for the non-shifted version clearly implies its tightness for the shifted one.
Theorem~\ref{thm:sLRZ_tight_3} shows that the converse is not true, since $(1, 3, 4, 6)$ is tight only for the shifted version.

The zonotopes associated with these three volume vectors are illustrated in Figure~\ref{fig:slrz_tight_last_covered},
where the last covered points of each have been highlighted in red and blue.
In all cases, the structure of the last covered points in a copy of $Z$ is the same: they consist of four parallel segments in four facets (two opposite pairs) of $Z$, plus two additional isolated points in two additional facets.
The way this configuration appears in the tiling $\tfrac35 Z + \ZZ^3$ is also the same in the three examples: all the segments are integer translates of one another, and each lies in the common intersection of four copies of $Z$.
The end-points of the segment lie each in an extra copy, and these six copies correspond to the six facets of $Z$ containing last covered points.
Put differently, the subsets $R_\pp$ in the proof of Lemma~\ref{lemma:last-covered} have four or five elements, depending on whether the last covered point lies in the interior of a segment or is the end-point of it.

% Tight sLR zonotopes
% (1, 2, 3, 4):
%    generators = [
%        [ 1, -1,  1],
%        [ 0,  1,  1],
%        [ 1,  1, -1],
%        [-1, -1,  0]]
%    last_covered_segment = [
%        (-3/5, -1/6, 8/15),
%        (-2/5,  1/6, 7/15)]

% (1, 3, 4, 6):
%    generators = [
%        [ 1,  1,  2],
%        [ 1, -1,  0],
%        [-1, -1,  1],
%        [ 0,  1, -1]]
%    last_covered_segment = [
%        ( 3/4,  11/20, 1/10),
%        (17/20,  9/20, 1/10)]

% (1, 3, 4, 7):
%    generators = [
%        [ 1,  0,  2],
%        [ 1, -1, -2],
%        [-1, -1,  1],
%        [ 0,  1,  0]]
%    last_covered_segment = [
%        (109/140, -29/35, -11/70),
%        ( 23/28,  -27/35, -17/70)]

\begin{figure}[htb]
    \centering
    \begin{minipage}{.33\textwidth}
        \centering
        \includegraphics[width=.95\textwidth]{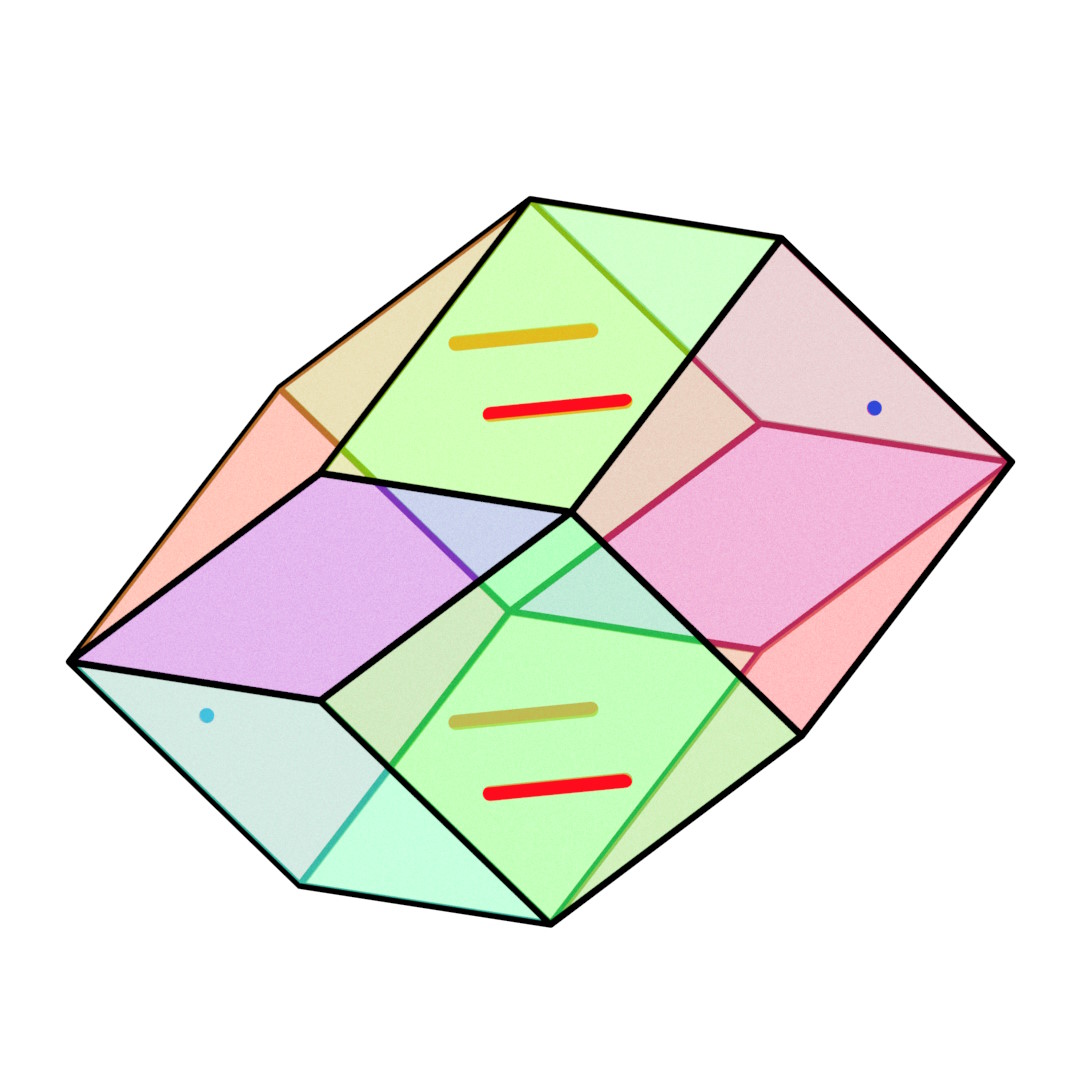}
    \end{minipage}\hfill%
    \begin{minipage}{.33\textwidth}
        \centering
        \includegraphics[width=.95\textwidth]{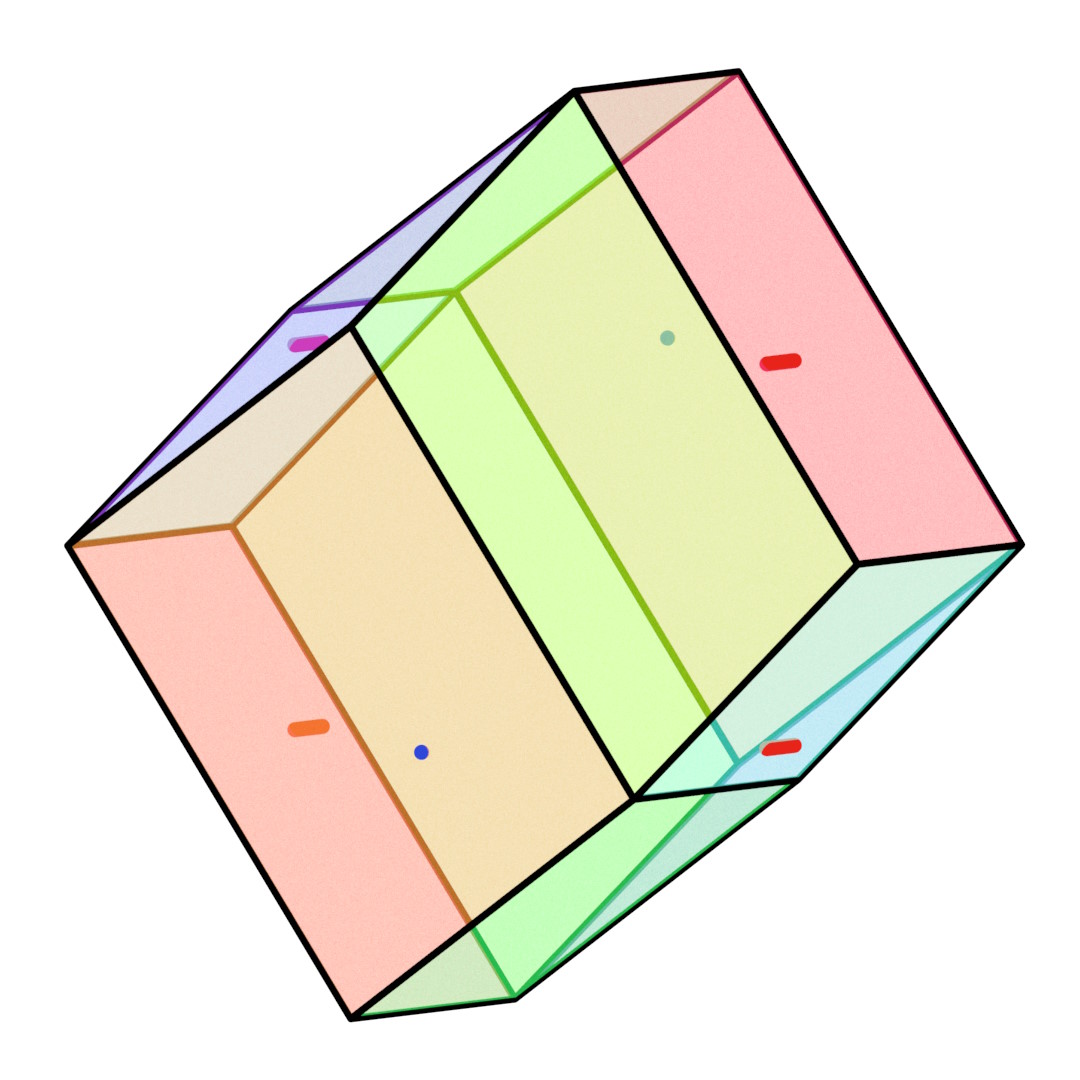}
    \end{minipage}\hfill%
    \begin{minipage}{.33\textwidth}
        \centering
        \includegraphics[width=.95\textwidth]{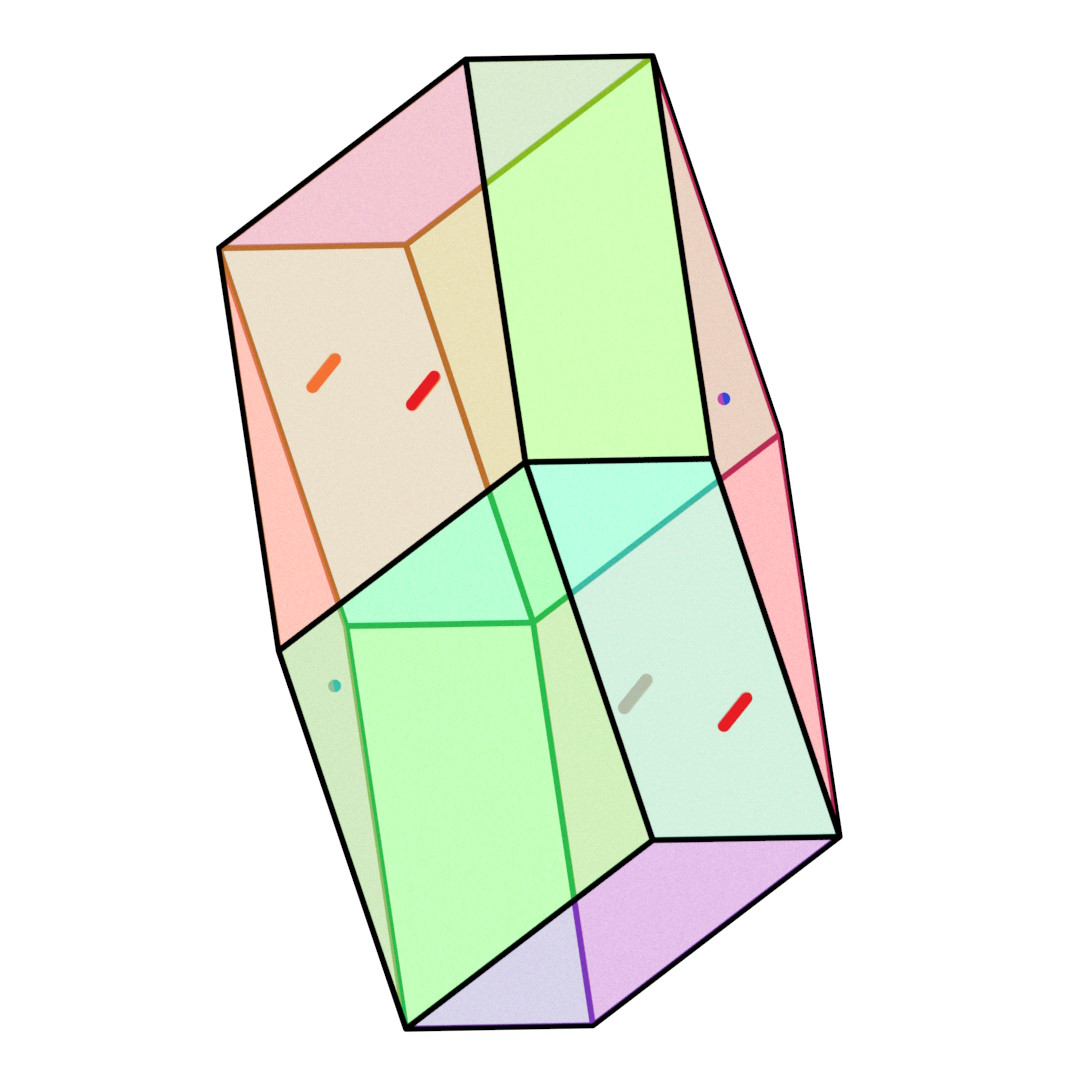}
    \end{minipage}
    \caption{The three sLR zonotopes with $\mu(Z) = \tfrac{3}{5}$ with their last covered points drawn in red (segments) and blue (points).}
    \label{fig:slrz_tight_last_covered}
\end{figure}

The projection of these zonotopes onto $\ZZ^2$ in the direction parallel to their last covered
segments is illustrated in Figure~\ref{fig:slrz_tight_last_covered_projections}.
The case $v=(1, 3, 4, 6)$ is notably distinct, as its last covered segment happens to be
parallel to one of its generators (the one associated with volume $3$).

\begin{figure}[htb]
    \begin{center}
        \begin{minipage}{.34\textwidth}
            \centering%
            \includegraphics[width=.95\textwidth]{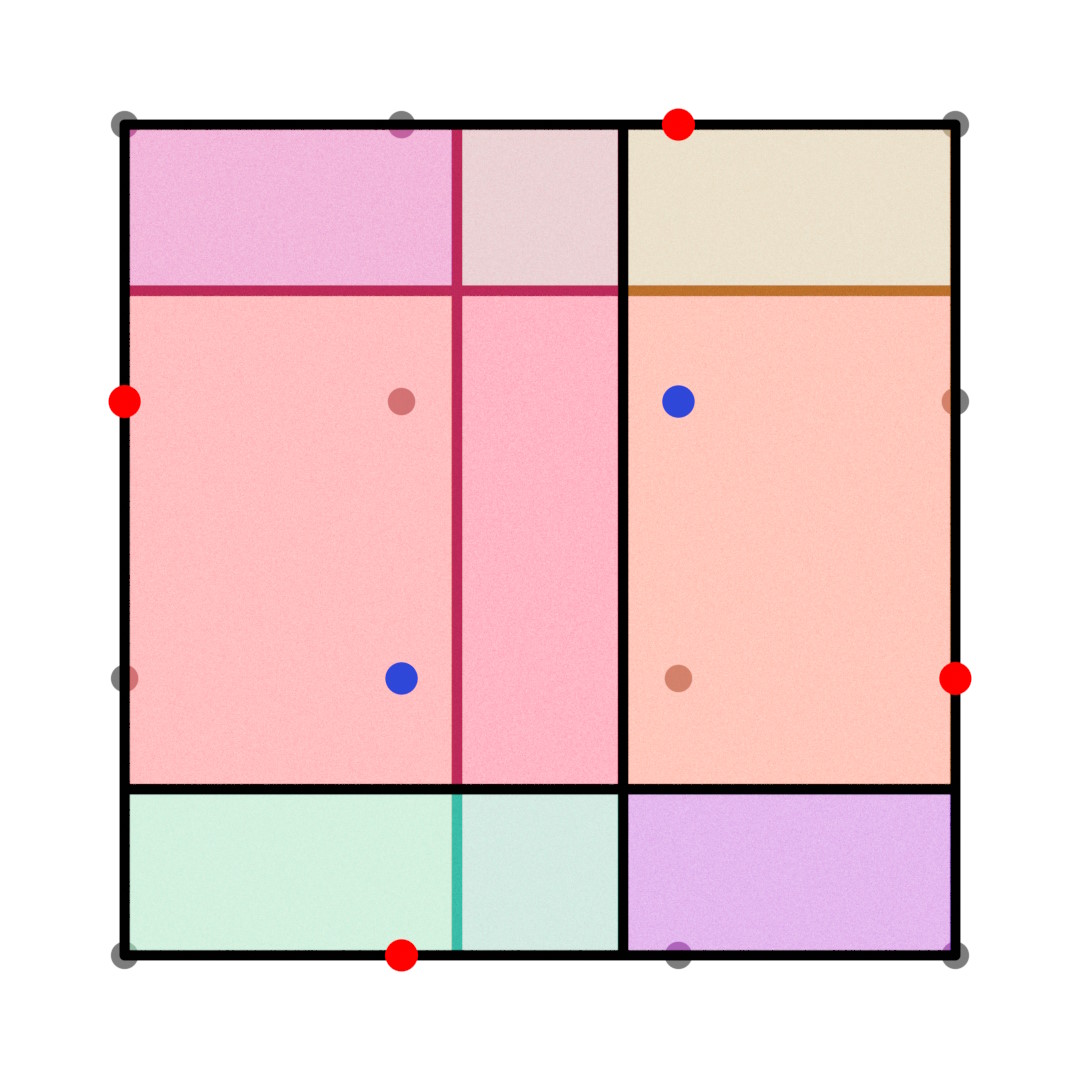}
        \end{minipage}\hfill%
        \begin{minipage}{.34\textwidth}
            \centering%
            \includegraphics[width=.95\textwidth]{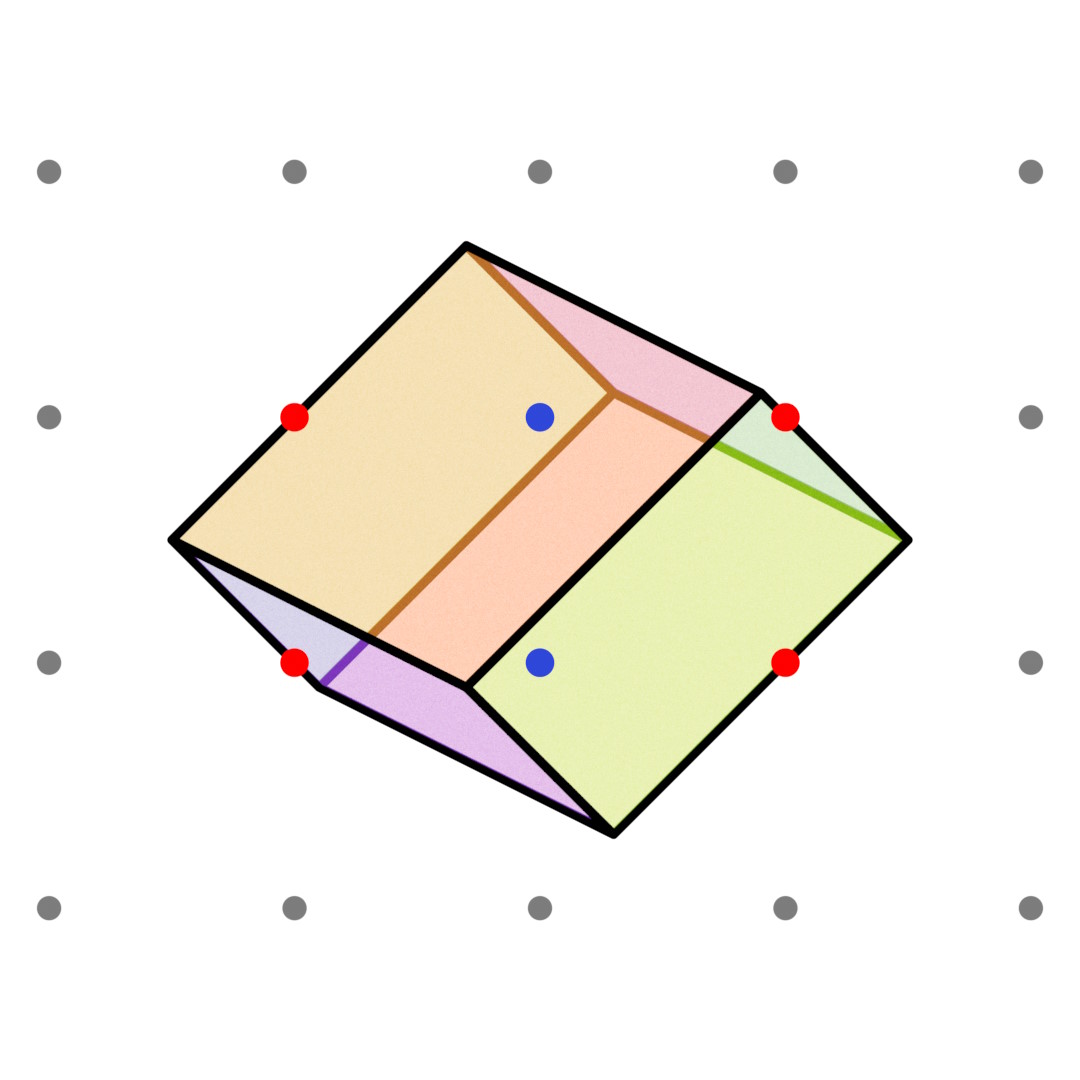}
        \end{minipage}\hfill%
        \begin{minipage}{.28\textwidth}
            \centering%
            \includegraphics[width=.75\textwidth,trim={10cm 0 10cm 0},clip]{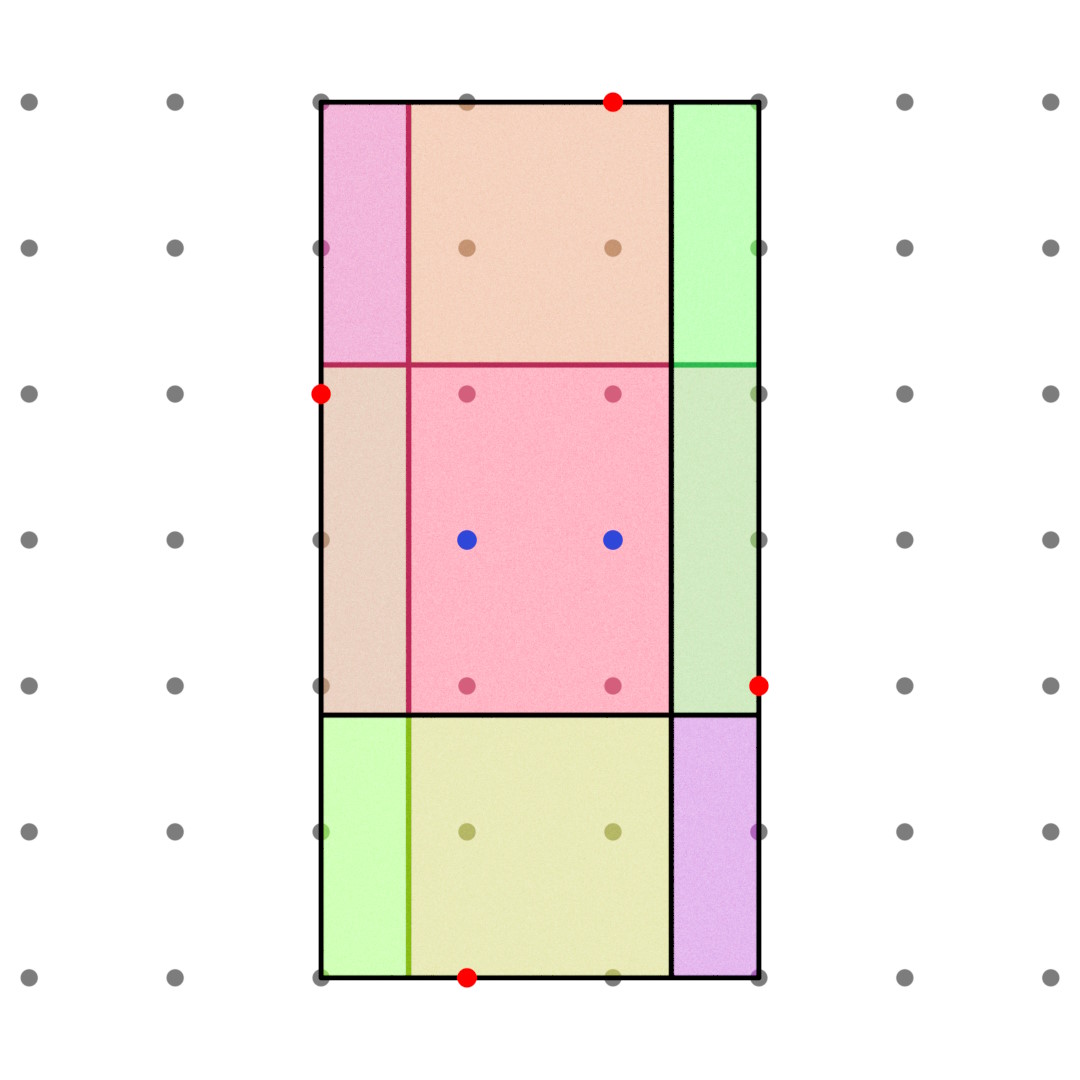}
        \end{minipage}
    \end{center}
    \caption{Projections of the three tight sLR zonotopes scaled down by $\mu = \tfrac{3}{5}$ in the direction of their last covered segments.
        The last covered segments project onto the dots in red, while the isolated last covered points project onto the dots in blue.}
    \label{fig:slrz_tight_last_covered_projections}
\end{figure}

\begin{remark}
\label{rem:last_covered}
Last covered points, considered modulo $\ZZ^3$ in the tiling $\tfrac35 Z + \ZZ^3$, correspond bijectively to
tight starting positions $(s_1,s_2,s_3, s_4)$, considered modulo $\ZZ$.
For example, the fact that the segment of last covered points for $v=(1, 3, 4, 6)$ is parallel to a generator in the zonotope (the generator associated with volume $3$) corresponds to the fact that only the runner with velocity $3$ has some freedom regarding its initial position in Proposition~\ref{prop:tight_1_3_4_6}.

In the cases $v=(1, 2, 3, 4)$ and $v=(1, 3, 4, 7)$ the last covered points also tell us that there is one degree of freedom in the set of tight starting positions, but now this degree of freedom is spread across two runners, which can be nudged
respecting a linear dependence.
\end{remark}

% \subsubsection{Statistical data}
% The count of sLR zonotopes, up to unimodular equivalence, that remain to be checked
% after filtering by different criteria are:
% \begin{enumerate}
%     \item With volume less than 196: 2.133.561
%     \item With volume less than 196, surviving Theorems F and G: 649.525
%     \item The above which do not contain a translated unit cube: 4.620
%     \item The above which do not contain a length 10 vertical segment, nor a unit cube transformed by an order 2 simple transform: 246
%     \item The above which do not contain a unit cube transformed by an order 3 simple transform: 245
%     \item The above which do not contain some specially chosen parallelepipeds of volume two: 244
%     \item The above which do not contain a parallelepiped of covering radius $\le \tfrac{3}{5}$: 121
% \end{enumerate}

% There are 5\,952\,978 sLR zonotopes of volume 1000, up to unimodular equivalence.

\subsection{Certificates and Reproduction}\label{subsec:certificates}
The certificate for our computational proof of Theorem~\ref{thm:sLRZ_195} is
the list of volume vectors from Proposition~\ref{prop:volume_vector_count}, each
accompanied by:
\begin{itemize}
    \item A set of \emph{small} generators for a sLR zonotope, $Z$, with the corresponding volume vector.
    \item A margin $\epsilon$, positive if we want to certify $\mu(P) \le \tfrac35$ and negative if we want to certify $\mu(P) < \tfrac35$.
          The positive ones need to be bounded by $\frac{1}{5D}$, where $D$ is a bound for the covering radius of $Z$
          as obtained from Corollary~\ref{cor:cov_radius_size_bound}, but the negative ones can be arbitrary (compare with Remark~\ref{rem:bound_in_negative_case}).
          All but three of them are negative, as proof that there are only three tight instances.
    \item A dyadic fundamental domain contained within the zonotope
          $(\tfrac{3}{5} + \epsilon) (Z - \cc)$, where $\cc$ is the center of $Z$.
\end{itemize}

Hence, the following steps certify the validity of our certificates:
\begin{enumerate}
    \item Verify that all volume vectors satisfying the conditions from
    Proposition~\ref{prop:volume_vector_count} are enumerated.
    \item For each volume vector:
    \begin{enumerate}
        \item Verify that the zonotope $Z$ has the stated volume vector.
              That is, compute the determinants in Definition~\ref{def:lrz_volume_vector}.
        \item Verify the validity of the margin $\epsilon$.
              That is, for the three positive ones verify that $\epsilon < \frac{1}{5D}$ where $D$ is any denominator bound
              for the covering radius of $Z$ according to Corollary~\ref{cor:cov_radius_size_bound}.
        \item Verify that all the voxels in the fundamental domain are contained in $(\tfrac{3}{5} + \epsilon) (Z - \cc)$.
    \end{enumerate}
\end{enumerate}

A mostly self-contained Python script with minimal dependencies that performs
these suggested steps under exact arithmetic is provided in the associated
\href{https://github.com/endorh/slrc-zonotopes/blob/main/check_certificates.py}{git repository}%
\footnote{\url{https://github.com/endorh/slrc-zonotopes/blob/main/check_certificates.py}}.
We have verified that all our certificates are valid.

\begin{remark}
The attentive reader may have noticed that for the three tight volume vectors our certificates only show $\mu(Z) \le \tfrac35$, but not $\mu(Z)\ge \tfrac35$ (i.e., tightness).
We do not need to certify tightness since: tightness of $(1,2,3,4)$ is obvious (see a proof before Theorem~\ref{thm:tight}), tightness of $(1,3,4,7)$ is proved in~\cite{cusickviewobIII}, and tightness of $(1,3,4,6)$ is proved in Proposition~\ref{prop:tight_1_3_4_6}.

For the non-tight volume vectors, our certificates show $\mu(Z) < \tfrac35$.
\end{remark}

\bibliography{mybib}{}
\bibliographystyle{plain}
\end{document}